\documentclass[11pt]{amsart}
\usepackage{pgf,tikz,pgfplots,color}
\pgfplotsset{compat=1.15}
\usepackage{mathrsfs} 
\usetikzlibrary{arrows}
\usepackage{fancyhdr}
\usepackage{lastpage} 

\usepackage[margin=1in]{geometry}
\usepackage{geometry}
\usepackage[toc,page]{appendix}
\usepackage[utf8]{inputenc}
\usepackage{hyperref}
\hypersetup{
	colorlinks=true,
	linkcolor=blue,
	citecolor=brown,
	filecolor=magenta, 
	urlcolor=blue,
}
\usepackage[notcite,notref]{} 
\usepackage{amsmath}
\usepackage{amsfonts}
\usepackage{amssymb}
\usepackage{amsthm}
\usepackage{setspace}
\usepackage{inputenc}
\usepackage[english]{babel} 
\usepackage{comment}
\usepackage[shortlabels]{enumitem}
\usepackage[T1]{fontenc}
\usepackage{esint}

\numberwithin{equation}{section}

\makeatletter
\def\@tocline#1#2#3#4#5#6#7{\relax
  \ifnum #1>\c@tocdepth 
  \else
    \par \addpenalty\@secpenalty\addvspace{#2}%
    \begingroup \hyphenpenalty\@M
    \@ifempty{#4}{%
      \@tempdima\csname r@tocindent\number#1\endcsname\relax
    }{%
      \@tempdima#4\relax
    }%
    \parindent\z@ \leftskip#3\relax \advance\leftskip\@tempdima\relax
    \rightskip\@pnumwidth plus4em \parfillskip-\@pnumwidth
    #5\leavevmode\hskip-\@tempdima
      \ifcase #1
       \or\or \hskip 1em \or \hskip 2em \else \hskip 3em \fi%
      #6\nobreak\relax
    \hfill\hbox to\@pnumwidth{\@tocpagenum{#7}}\par
    \nobreak
    \endgroup
  \fi}
\makeatother

\title[]{A unique continuation property for $|\overline \partial u| \leq V |u|$}           

\author[]{Ziming Shi}

\address{Department of Mathematics,
	University of California-Irvine, Irvine, CA, 92697} 
\email{zimings3@uci.edu}

\subjclass[2020]{ }  

\newcommand{\dist}{\operatorname{dist}}

\newcommand{\supp}{\operatorname{supp}}

\newcommand{\loc}{\mathrm{loc}}
\newcommand{\Res}{\operatorname{Res}} 

\newtheorem{thm}{Theorem}[section]
\newtheorem{cor}[thm]{Corollary} 
\newtheorem{prop}[thm]{Proposition}
\newtheorem{lemma}[thm]{Lemma}

\theoremstyle{definition}
\newtheorem{defn}[thm]{Definition}
\newtheorem{exmp}[thm]{Example}
\newtheorem{ques}[thm]{Question}

\theoremstyle{remark}
\newtheorem{rem}[thm]{Remark}
\newtheorem*{clm}{Claim}
\newtheorem*{ack}{Acknowledgment}

\renewcommand{\th}[1]{\begin{thm}\label{#1}}
	\renewcommand{\eth}{\end{thm}}
\newcommand{\co}[1]{\begin{cor}\label{#1}}
	\newcommand{\eco}{\end{cor}}
\newcommand{\pr}[1]{\begin{prop}\label{#1}}
	\newcommand{\epr}{\end{prop}}

\newcommand{\df}[1]{\begin{defn}\label{#1}}
	\newcommand{\edf}{\end{defn}}
\newcommand{\ex}[1]{\begin{exmp}\label{#1}} 
	\newcommand{\eex}{\end{exmp}}
\newcommand{\qu}[1]{\begin{ques}\label{#1}}
	\newcommand{\equ}{\end{ques}}  
\newcommand{\mk}{\begin{rem}}
	\newcommand{\emk}{\end{rem}}
\newcommand{\cl}{\begin{clm}}
	\newcommand{\ecl}{\end{clm}} 
\newcommand{\ac}{\begin{ack}}
	\newcommand{\eac}{\end{ack}} 

\newcommand{\ga}{\begin{gather}}
\newcommand{\ega}{\end{gather}}
\newcommand{\gan}{\begin{gather*}}
\newcommand{\egan}{\end{gather*}}
\newcommand{\al}{\begin{gngn}}
	\newcommand{\eal}{\end{align}}
\newcommand{\aln}{\begin{align*}}
\newcommand{\ealn}{\end{align*}}
\newcommand{\eq}[1]{\begin{equation}\label{#1}}
\newcommand{\eeq}{\end{equation}}

\newcommand{\pa}{\partial{}}
\newcommand{\na}{\nabla}
\newcommand{\db}{\dbar}

\newcommand{\we}{\wedge}

\newcommand{\ra}{\longrightarrow}

\newcommand{\sm}{\setminus}
\newcommand{\seq}{\subseteq}

\newcommand{\Z}{\mathbb{Z}}
\newcommand{\R}{\mathbb{R}} 
\newcommand{\C}{\mathbb{C}}
\newcommand{\D}{\mathbb{D}}

\newcommand{\ov}{\overline}
\newcommand{\ti}{\tilde}
\newcommand{\wti}{\widetilde}
\newcommand{\hht}{\widehat}

\newcommand{\RE}{\operatorname{Re}}

\newcommand{\dbar}{\overline\partial}

\newcommand{\all}{\alpha}

\newcommand{\del}{\delta}
\newcommand{\Del}{\Delta}
\newcommand{\var}{\varphi}
\newcommand{\e}{\epsilon}
\newcommand{\ve}{\varepsilon}
\newcommand{\om}{\omega}
\newcommand{\Om}{\Omega}
\newcommand{\thh}{\theta}
 
\newcommand{\La}{\Lambda}
\newcommand{\la}{\lambda}

\newcommand{\gm}{\gamma}

\newcommand{\si}{\sigma}

\newcommand{\yh}{\frac{1}{2}}

\newcommand{\re}[1]{(\ref{#1})}

\newcommand{\rl}[1]{Lemma~\ref{#1}}

\newcommand{\rp}[1]{Proposition~\ref{#1}}
\newcommand{\rt}[1]{Theorem~\ref{#1}}

\newcommand{\nid}{\noindent}

\newcounter{pp}
\newcommand{\bpp}{\begin{list}{$\hspace{-1em}\alph{pp})$}{\usecounter{pp}}}
	\newcommand{\epp}{\end{list}}

\newcounter{ppp}
\newcommand{\bppp}{\begin{list}{$\hspace{-1em}(\roman{ppp})$}{\usecounter{ppp}}}
	\newcommand{\eppp}{\end{list}}

\newcommand{\Ac}{\mathcal{A}}

\begin{document} 
	\definecolor{rvwvcq}{rgb}{0.08235294117647059,0.396078431372549,0.7529411764705882}
\maketitle  

\begin{abstract} 
Let $u: \Omega \subset \mathbb C^n \to \mathbb C^m$, for $n \geq 2$ and $m \geq 1$. Let $1 \leq p \leq 2$, and $2(2n)^2 -1 \leq q < \infty$ such that $\displaystyle \frac{1}{p} + \frac{1}{p'} = 1$ and $\displaystyle \frac{1}{p} - \frac{1}{p'} = \frac{1}{q}$. Suppose $|\overline \partial u| \leq V |u|$, where $V \in L^q_{\operatorname{loc}}(\Omega)$. Then $u$ has a unique continuation property in the following sense: if $u \in W^{1,p}_{\operatorname{loc}}(\Omega)$ and for some $z_0 \in \Omega$, $\| u \|_{L^{p'}(B(z_0,r))} $ decays faster than any powers of $r$ as $r \to 0$, then $u \equiv 0$. The same result holds for $q=\infty$ if $u$ is scalar-valued ($m=1$).   
\end{abstract} 
\vspace{2cm}

\tableofcontents 

\section{Introduction} 
\begin{defn}
 Let $\Om$ be an open set in $\R^d$ and let $z_0 \in \Om$. For $u \in L^p_\loc(\Om)$, we say that $u$ \emph{vanishes to order $N$ in the $L^p$ sense at $z_0$} if 
\begin{equation} \label{ord_N_vanish}  
   \lim_{r \to 0} r^{-N} \int_{B(z_0,r)} |u|^p = 0.    
\end{equation}
We say that $u$ \emph{vanishes to infinite order in the $L^p$ sense at $z_0$} if \re{ord_N_vanish} holds for all $N$. 
\end{defn}
The following notation will be used throughout the paper. 
\[
  \frac{1}{p} + \frac{1}{p'} =1, \quad \frac{1}{p} - \frac{1}{p'} = \frac{1}{q}, \quad 1 \leq p \leq 2. 
\]
The purpose of this paper is to establish a strong unique continuation property (SUCP) for the differential inequality $|\db u| \leq V |u|$, where the potential is locally in $L^q$ for suitable $q$. 
\begin{thm} \label{Thm::mt_intro}  
Let $\Om$ be a connected open subset of $\C^n$, $n \geq 2$. Let $2(2n)^2 -1 \leq q \leq \infty$ and set $p = \frac{2q}{q+1}$, $p' = \frac{2q}{q-1}$ so that $\frac{1}{p} - \frac{1}{p'} = \frac{1}{q}$ (in particular $p=2$ when $q = \infty$). Then the differential inequality $|\db u| \leq V|u| $ has the following strong unique continuation property: 
\begin{enumerate}[(i)] 
 \item 
Let $u: \Om \to \C^m$ for $m \geq 1$, and let $|\db u| \leq V|u|$ with $V \in L^q_\loc(\Om)$, for some $2(2n)^2-1 \leq q <\infty$. 
If $u \in W_\loc^{1,p}(\Om)$ where $p = \frac{2q}{q+1} \in (1,2)$, and $u$ vanishes to infinite order in the $L^{p'}$ sense at some $z_0 \in \Om$, then $u$ is $0$ in $\Om$.   
\item 
Let $u: \Om \to \C$ and let $|\db u| \leq V|u|$ with $V \in L^\infty_\loc(\Om)$. If $u \in W_\loc^{1,2}(\Om)$, and $u$ vanishes to infinite order in the $L^{2}$ sense at some $z_0 \in \Om$, then $u$ is $0$ in $\Om$.   
\end{enumerate}

\end{thm}
Our proof for part (i) and (ii) are very different. Part (i) is proved using the Carleman inequality method, whereas part (ii) relies on a theorem of Gong and Rosay \cite{G-R_07}, which states that for a continuous function $f: \Om \subset \C^n \to \C$ that is not identically zero, if $|\db f| \leq C |f|$ for some positive constant $C$, then the zero set of $f$ is a complex analytic set and thus has real dimension at most $2n-2$. For (i) we can prove the SUCP for any complex vector-valued functions, due to the fact that Carleman inequality works equally well for scalar-valued and vector-valued functions. For (ii) however, we can only prove for $\C$-valued functions, since there is currently no vector-valued analog for the theorem of Gong and Rosay. We note that part (ii) is not new and actually contained in a more general result proved recently by Pan and Zhang \cite[Theorem 5.1]{P-Z_24}, where they assume that $V$ takes the form $\frac{C}{|z|}$, however we decide to include our proof here since it is slightly simpler (due to our stronger assumption $V \in L^\infty_\loc$).  


When the source domain is 1-dimensional (i.e. $\Om \subset \C$), Pan and Zhang \cite{P-Z_23} proved that SUCP holds for $|\db u| \leq V |u|$ with $V \in 
L^2_\loc(\Om)$, for functions $u: \Om \to \C^m$ with $u \in W^{1,2}_{\loc}(\Om)$. The $L^2_\loc$ condition is optimal, as SUCP fails if $V \in L^p_\loc(\Om)$ for any $p<2$, by the example $u=e^{-|x|^{-\ve}}$ with suitably chosen $\ve$. In fact, the same example also shows that for $u: \C^n \to \C^m$ with $n \geq 2$, SUCP fails if $V \in L^q_\loc$ for any $q<2n$.

Inequalities involving the $\db$ operator arise naturally in complex analysis and geometry. In their study of boundary uniqueness of holomorphic mappings between smooth hypersurfaces, Bell and Lempert \cite{B-L_90} showed that given $u: \D \to \C^n$, $u \in C^\infty(\D)$, and $u$ satisfies $| \db u | \leq C |u|$ for some $C>0$, if $u$ vanishes to infinite order (all its derivatives vanish) at $0$, then $u$ vanish identically in a neighborhood of $0$. 
For applications of unique continuation of $\db$ inequalities in $J$ holomorphic curves and other related problems in almost complex analysis, the reader may refer to \cite{Ros10}.  

For our proof we follow the method of Wolff \cite{Tom90}, who considered a similar SUCP problem for the differential inequality $|\Del u| \leq V |\na u|$. We started out by constructing a kernel $L_\nu(z,\zeta)$ which is a $(n,n-1)$ form on $\C^n \times \C^n$, so that  
\[
  |z|^{-\nu} u(z) = \int L_\nu(z,\zeta) \we |\zeta|^{-\nu} \db u(\zeta), \quad u \in C^\infty_c(\C^n \sm \{ 0 \}).   
\]
We then osculate the weight $|z|^{-\nu}= e^{\nu \log \frac{1}{|z|}}$ to $e^{\nu \psi(\log \frac{1}{|z|})}$, where $ \psi$ is a strictly convex function, and obtain the following 
\[
 e^{\nu \psi (\si)} u (z) = \int P_\nu (z,\zeta) e^{\nu \psi(\tau)} \we \db u(\zeta), \quad \si = \log \frac{1}{|z|}, \quad \tau = \log \frac{1}{|\zeta|}.  
\]
This allows us to localize in the radial variable to intervals of logarithmic length $\nu^{-\yh}$ and obtain a Carleman inequality of the form: 
\begin{equation} \label{Car_ineq_intro} 
  \| e^{\nu \psi (\si)} u \|_{L^{p'} (A(\psi^{-1} \gm))} \leq C_n \nu \min 
  \{ |\gm|,\nu^{-\yh}\} \|  e^{\nu \psi (\si)} \db u \|_{L^{p} (B(0,1))}.    
\end{equation}
  
Here $\gm \subset \R$ is an interval in the radial variable and $A(\psi^{-1} \gm)$ is the annulus $\{ z \in \C^n: \si(z) = \log (1/|z|) \in \psi^{-1}(\gm) \}$.  
Note that for fixed $\gm$, the quantity $\nu \min \{ |\gm|,\nu^{-\yh}\} $ blows up as $\nu \to \infty$, and hence \re{Car_ineq_intro} is not considered a true Carleman inequality and one cannot apply the standard argument to derive unique continuation property directly from \re{Car_ineq_intro}. Instead, we use the following key insight of Wolff: for each $\nu$, there exists a collection of disjoint radial intervals $\gm_j$ such that $\left| e^{\nu \psi(\si)} V |u| \right|^p$ is concentrated on $A^{-1}(\gm_j)$ for each $j$, while $\gm_j$-s are 
``collectively small" on the scale of $\nu^{-1}$. One can then apply \re{Car_ineq_intro} on each such annulus and take the sum to get a true Carleman inequality with a constant that does not blow up with $\nu$, from which unique continuation easily follows.  

Since $\db$ is a first-order differential operator, and in view of the counterexample above, it is natural to conjecture that $L^{2n}$ (or perhaps $L^{2n+\ve}$) is the optimal condition on $V$ for SUCP to hold (See also Wolff \cite{Tom93} for a discussion of related unique continuation problems for second-order elliptic operators.) We list this as an open problem. 
\begin{ques} \label{Ques::L^2N} 
 Let $\Om$ be a connected open subset of $\C^n$ and let $u: \Om \to \C^m$. Set $p= \frac{4n}{2n+1}$ and $p'= \frac{4n}{2n-1}$ so that $\frac{1}{p}-\frac{1}{p'} = 
\frac{1}{2n}$. Suppose $u \in W^{1,p}_\loc(\Om)$ and satisfies $|\db u| \leq V |u|$, with $V \in L^{2n}_\loc(\Om)$. If $u$ vanishes to infinite order in the $L^{p'}$ sense at some point $z_0 \in \Om$, does $u$ vanish identically? 
\end{ques}
Another question is whether \rt{Thm::mt_intro} (ii) still holds for vector-valued functions. 
\begin{ques} \label{Ques::L^infty}  
 Let $\Om$ be a connected open subset of $\C^n$ and let $u: \Om \to \C^m$, for $m \geq 2$. Suppose $u \in W^{1,2}_\loc(\Om)$ and satisfies $|\db u| \leq V |u|$, with $V \in L^{\infty}_\loc(\Om)$. If $u$ vanishes to infinite order in the $L^{2}$ sense at some point $z_0 \in \Om$, does $u$ vanish identically? 
\end{ques} 
We remark that if $u$ is smooth, then the infinite order vanishing of $u$ at $z_0$ means that $|u(z)| \leq C|z-z_0|^N$ for any $N > 0$, or equivalently, all the derivatives of $u$ vanish at $z_0$. 
In the recent preprint \cite{P-Z_24}, Pan and Zhang showed that if $u: \Om \subset \C^n  \to \C$ with $u \in C^\infty_\loc(\Om)$, $|\db u| \leq V|u|$ for some 
$V \in L^{2n}_\loc(\Om)$, and $u$ vanishes to infinite order at $z_0$, then $u$ is $0$. Furthermore, they proved that the same result holds for vector ($\C^m$) valued smooth functions if $V \in L^{2n+\ve}$ for any $\ve>0$. In the proof of both these results, the smoothness assumption of $u$ is used to extend the flatness of $u$ into its restrictions along complex one-dimensional radial directions, so then one only needs to solve a 1-dimensional SUCP problem.   

We now fix some notations used in the paper.  
We denote by $C^\infty_c(D)$ the space of $C^\infty$ functions with compact support in $D$. 
We write $x \lesssim y$ to mean that $x \leq Cy$ for some constant $C$ independent of $x$ and $y$, and we write $x \approx y$ if $x \lesssim y$ and $y \lesssim x$. By $D^l$ we mean a differential operator of order $l$: $D^l_z g(z)= \pa_{z_i}^{\all_i} \pa_{\ov z_j}^{\beta_j} g(z)$, $\sum_i \all_i + \sum_j \beta_j = l$. 
\section{Taylor expansion of the fundamental solution}  
For $z,\zeta \in \C^n$, $\zeta \neq 0$, denote 
\[ 
 r = \frac{|z|}{|\zeta|}, \quad \thh =\thh_{z\zeta} = \angle_{z0\zeta} \in [0, \pi]. 
\] 
For a fixed $\zeta$, we denote by $T^N_z f(z,\zeta)$ the degree $N$ Taylor polynomial of $f(\cdot,\zeta)$ at $z=0$. Similarly, for fixed $\thh$, we let $T^N_r g(r,\theta)$ be the degree $N$ Taylor polynomial of $g(\cdot,\thh)$ at $r=0$. We write $d=2n$, and for all the computations we will assume $n \geq 2$ and $d \geq 3$.  

First we need a fact about the Taylor expansion of the Newtonian potential. The result is used in \cite{Tom90} and we shall provide the details. 
\begin{prop} \label{Prop::ul_sp_est}  
 Fix $\zeta \neq 0$. We have the following Taylor expansion 
\[
   \frac{1}{|z-\zeta|^{d-2}} = \frac{1}{|\zeta|^{d-2}} \sum_{m=0}^\infty P^{(\frac{d-2}{2})}_m(\cos \thh) r^m,    
\]
where $P_m^{(\frac{d-2}{2})}$ is an ultraspherical polynomial of degree $m$ such that $|P^{(\frac{d-2}{2})}_m(\cos \thh)| \leq m^{d-3}$.   
\end{prop}
\begin{proof}
Denote $\la = \frac{d-2}{2}$. We have 
\[
  \frac{1}{|z-\zeta|^{d-2}} = \frac{1}{(|z -\zeta|^2)^\la} = \frac{1}{\left(|z|^2 - 2 |z| |\zeta| \cos \thh + |\zeta|^2 \right)^\la} = \frac{1}{|\zeta|^{d-2} \left( 1 - 2 r \cos \thh  + r^2 \right)^\la }.  
\]
Expand $(1-2 \tau \mu + \mu^2)^{-\la} $ as a power series of $r$ at $r = 0$ to get  
\begin{equation} \label{Pi_expansion} 
 (1-2 \tau \mu + \mu^2)^{-\la} = P^{(\la)}_0 (\cos \thh) + P^{(\la)}_1 (\cos \thh) r + P^{(\la)}_2 (\cos \thh) r^2 + \cdots . 
\end{equation} 
On the other hand we can write 
\[
  (1-2 \tau \mu + \mu^2)^{-\la} = (1 - e^{i \thh} r)^{-\la} (1-e^{-i \thh} r)^{-\la} 
\]
Expanding $1 - e^{i \thh} r $ and $1 - e^{i \thh} r $ in binomial series, we get 
\begin{gather*}
(1 - e^{i \thh} r)^{\la} = \all^{(\la)}_0 + \all^{(\la)}_1 e^{i\thh} r + \all^{(\la)}_2 e^{2i \thh }r^2 + \cdots, \quad 
 (1 - e^{-i \thh} r)^\la = \all^{(\la)}_0 + \all^{(\la)}_1 e^{-i\thh} r + \all^{(\la)}_2 e^{-2i \thh }r^2 + \cdots.  
\end{gather*}
Hence 
\begin{equation} \label{alli_expansion} 
  (1 - e^{i \thh} r)^{\la} = \all^2 + (2 \all_0 \all_1 \cos \thh) r + (2 \all_0 \all_2 \cos 2 \thh + \all_1^2) r^2 + (2 \all_0 \all_3 \cos 3 \thh + 2 \all_1 \all_2 \cos \thh) r^3 + \cdots.    
\end{equation}
Comparing the coefficients in \re{Pi_expansion} and \re{alli_expansion}, we get 
the following formula for $P^{(\la)}_k$:
\[
  P^{(\la)}_k = 2 \all_0 \all_k \cos k \thh + 2 \all_1 \all_{k-1} \cos (k-2) \thh + 2 \all_2 \all_{k-2} \cos (k-4)\thh +  \cdots,  
\]
the last term being $ \all^2_{\frac{k}{2}}$ if $k$ is even, and $2 \all_{\frac{k-1}{2}} \all_{\frac{k+1}{2}} \cos \thh$ if $k$ is odd. Now the binomial coefficients $\all^{(\la)}_k$ are given by (cf. \cite[p.~93]{Sze75}) 
\[
  \all^{(\la)}_k = \begin{pmatrix}
      k + \la -1 \\ k
  \end{pmatrix}, \quad k = 0, 1, 2, \dots.
\]
If $\la >0$ (which is true in our case since $\la = \frac{d-2}{2} >0$ when $d \geq 3 $), then $\all^{(\la)}_k $ are all positive. Now $P^{(\la)}_k $ is a trig polynomial in $\cos \thh$ (i.e. a linear combination of $\cos \thh, \cos 2\thh, \dots)$ with positive coefficients, thus it must achieve its maximum at $\cos 0 = 1$: 
\[
   | P_k^\la (\cos \thh)| \leq P_k^\la (1) = \begin{pmatrix}
      k + 2\la -1 \\ k
  \end{pmatrix} 
\approx k^{2\la -1} = k^{d-3}, 
\]
where we use the formula for $P_k^\la (1)$ from \cite[p.~80]{Sze75}. 
\end{proof}
\begin{prop} \label{Prop::IN_hg_form} 
For each non-negative integer $N$, there exists a $(n,n-1)$ form $I^N(z,\zeta)$ such that 
\[
  |z|^{-N} u(z) = \int_{\C^n} I^N(z,\zeta) |\zeta|^{-N+d-1} \we \db u(\zeta), \quad u \in C^\infty_c(\C^n \sm \{ 0 \}), 
\] 
where $I^N(z,\zeta) = c_n \sum_{j=1}^n I_j^N(z,\zeta) \hht{d\ov \zeta_j} $ and 
\begin{equation} \label{IN_hg_form}
  I_j^N(z,\zeta) = c_d h_{j1}(z,\zeta) r^{-N} \left( g(r,\thh) - T_r^{N-1} [g](r,\thh) \right)   
 +c_d  h_{j2}(z) r^{-(N-1)} \left( g(r,\thh) - T_r^{N-2} [g](r,\thh) \right).
\end{equation}
The functions $g, h_{j1}, h_{j2}$ are given by 
\begin{gather*}
  g(r,\thh) =  \left[(r-e^{i\thh})(r-e^{-i \thh}) \right]^{-\frac{d}{2}};
  \\ 
   h_{j1}(z,\zeta) = - \frac{1}{|\zeta|} \left( \ov \zeta_j - \yh \frac{ \ov z_j}{|z|^2} \left( z\cdot \ov \zeta + \ov z \cdot \zeta \right) \right) - \frac{|z|^2 + |\zeta|^2 - |z-\zeta|^2}{2|z| |\zeta|} \cdot \frac{\ov z_j}{2 |z|}; \quad 
  h_{j2}(z) = \frac{\ov z_j}{2 |z|}. 
\end{gather*}
In particular, the $h_{j1}$ and $h_{j2}$ are homogeneous in $(z,\zeta)$ of degree $0$. 
\end{prop}
\begin{proof} 
Let $u \in C^\infty_c(\C^n \sm \{0\})$. We can write 
\[
 u(z) = \int_\Om B(z,\zeta) \we \db u (\zeta),  
\]
where $B(z,\zeta)$ is the Bochner-Martinelli formula given by 
\[
B(z,\zeta) = c_{n}
\sum_{j=1}^n B_j(z,\zeta) \hht{d \ov \zeta_j} , \quad c_n = 
\frac{(n-1)!}{(2 \pi i)^n}, \quad B_j(z,\zeta) = \frac{\ov {\zeta_j - z_j}}{|\zeta-z|^{2n}} 
\]
for $(z,\zeta) \in (\C^n \times \C^n) \sm\{ \zeta =z \} $. 
Here $\hht{d \ov \zeta_j}:= d \zeta_j 
\we \left( \bigwedge_{k \neq j} d \ov \zeta_k \we d \zeta_k \right) $. 

Given a non-negative integer $N$ and a fixed $\zeta \neq 0 $, let $P_j^N(\cdot,\zeta)$ be the degree $N$ Taylor polynomial of $B_j(z,\zeta)$ at $z=0$, and let 
$P^N (z,\zeta) = c_{2n} \sum_{j=1}^n P_j^N(z,\zeta) \hht{d\ov \zeta_j} $ . We have $\db_\zeta B(z,\zeta) = 0$ for $\zeta \neq z$, or
\[
 \sum_{j=1}^n \pa_{\ov \zeta_j} B_j(z,\zeta) = 0, \quad \zeta \neq z.  
\]
Since 
\[ 
D^\beta_\zeta (D_z^\all B_j(0,\zeta)) = (D^\beta_\zeta D_z^\all B_j)(0,\zeta)
= (D_z^\all D^\beta_\zeta B_j)(0,\zeta), 
\]
we have for $\zeta \neq 0$,  
\begin{align*}
\sum_{j=1}^n \pa_{\ov \zeta_j} P_j^N (z,\zeta) 
&= \sum_{j=1}^n \pa_{\ov \zeta_j} \sum_{|\all| + |\beta| \leq N}  \frac{D^\all_z B_j (0,\zeta)}{\all !} z^\all \ov z^\beta  
\\ &= \sum_{|\all| + |\beta| \leq N} \left( \left[ D_{z}^\all \sum_{j=1}^n  \pa_{\ov \zeta_j} B_j (z,\zeta) \right] (0,\zeta) \right) z^\all \ov z^\beta = 0, 
\end{align*}
i.e. $\db_\zeta P^N (z,\zeta) = 0$, for any $z \in \C^n$ and $\zeta \in \supp u$. Using Stokes theorem, we can write 
\[
  u(z) = \int_\Om \left[ B(z,\zeta) - P_{N-1}(z,\zeta) \right] \we \db u(\zeta)
  = c_{2n} \sum_{j=1}^n \int_\Om 
\left[B_j(z,\zeta) - P_j^{N-1}(z,\zeta) \right] \hht{d \ov \zeta_j} \we \db u (\zeta).  
\]
Define 
\[
 I_j^N(z,\zeta) = \left( \frac{|\zeta|}{|z|} \right)^{ N } |\zeta|^{ {d-1 } } \left( B_j (z,\zeta) - P_j^{N-1}(z,\zeta) \right),   
\]
and $I^N(z,\zeta) = c_d \sum_{j=1}^n I_j^N(z,\zeta) \hht{d\ov \zeta_j} $. 
Then we can write 
\[
  |z|^{-N} u(z) = \int_\Om I^N(z,\zeta) |\zeta|^{-(N+d-1)} \we \db u(\zeta). 
\]
We now write $B_j = \frac{\ov {\zeta_j - z_j}}{|\zeta-z|^d} 
=c_d \pa_{z_j} \left( \frac{1}{|\zeta-z|^{d-2}} \right)$, for $c_d = \frac{2-d}{2} $ and 
\[ 
 |\zeta|^{d-2} B_j (z,\zeta) 
 = c_d \pa_{z_j} \left( \frac{|\zeta|^{d-2}}{|\zeta-z|^{d-2}} \right) 
 =  c_d \pa_{z_j}  \left( \frac{|\zeta-z|}{|\zeta|} \right)^{-(d-2)}. 
\] 
Since $|\zeta|^{d-2} P_j^{N-1}(\cdot,\zeta)$ is the degree $N-1$ Taylor polynomial of $|\zeta|^{d-2}B_j(\cdot,\zeta)$ at $z = 0$, by the previous line we have 
\[
  |\zeta|^{d-2} P_j^{N-1} = c_d \pa_{z_j} T^N_z \left[ \left( \frac{|\zeta-z|}{|\zeta|} \right)^{-(d-2)} \right].  
\]
Now 
\[
  \left( \frac{|\zeta-z|}{|\zeta|} \right) ^{-(d-2)}=  \left( \frac{|\zeta-z|^2}{|\zeta|^2}  \right)^{-\frac{d-2}{2}} 
  = \left[(r-e^{i\thh})(r-e^{-i \thh}) \right]^{-\frac{d-2}{2}} := f(r,\thh)  . 
\]
Combining the above expression, we get 
\begin{equation} \label{IN_as_thh_r_der}   
\begin{aligned}  
   I_j^N (z,\zeta) 
&= r^{-N} |\zeta| \left[ |\zeta|^{d-2} B_j (z,\zeta) - |\zeta|^{d-2} P_j^{N-1}(z,\zeta)  \right]
\\ &=  c_d r^{-N} |\zeta| \left[ \pa_{z_j} \left( f(r,\thh)  - T^N_r f(r,\thh)  \right) \right] 
\\ &= c_d r^{-N} |\zeta| 
\left[ (\pa_{z_j} r) \pa_r \left( f(r,\thh) - T^N_r f(r,\thh) \right) + (\pa_{z_j} \thh) \pa_\thh \left( f(r,\thh)  - T^N_r f(r,\thh)  \right) \right]. 
\end{aligned}  
\end{equation} 
Here $T^N_r f(r,\thh)$ denotes the degree $N$ Taylor polynomial of $f(\cdot, \thh)$ at $r=0$.   

Now we have 
\begin{align*}
   \pa_r f(r,\thh) &= \pa_r \left[(r-e^{i\thh})(r-e^{-i \thh}) \right]^{-\frac{d-2}{2}}
   \\ &= - \frac{d-2}{2} \left[(r-e^{i\thh})(r-e^{-i \thh}) \right]^{-\frac{d}{2}} 
   \cdot \pa_r \left( r^2 - 2 r \cos \thh  + 1 \right) 
   \\ &= (2-d) (r-\cos \thh) g(r,\thh),  
\end{align*}
where we denote 
\[
  g(r,\thh) =  \left[(r-e^{i\thh})(r-e^{-i \thh}) \right]^{-\frac{d}{2}}. 
\] 
Similarly, we get 
\begin{align*}
 \pa_\thh f(r,\thh) = (2-d) (r \sin \thh) g(r,\thh).   
\end{align*} 
The corresponding Taylor expansions are 
\begin{align*}
\pa_r T^N_r f(r,\thh)  
&= T^{N-1}_r \pa_r f(r,\thh) 
\\ &= (2-d) T^{N-1}_r \left[ (r-\cos \thh) g(r,\thh) \right]  
\\ &= (2-d) \left[ r T^{N-2}_r g(r,\thh) - \cos \thh \cdot T^{N-1}_r g(r,\thh) \right],    
\end{align*}
and 
\begin{align*}
  \pa_\thh T^N_r f(r,\thh) 
 &= T^N_r [\pa_\thh f](r,\thh)
 \\&= T^N_r [(2-d) (r\sin \thh) g(r,\thh)](r,\thh)
 \\ &= (2-d) (r \sin \thh) \cdot T^{N-1}_r[g](r,\thh).
\end{align*}
Plug the expressions into \re{IN_as_thh_r_der} to get
\begin{equation} \label{IN_inter}  
\begin{aligned} 
    I_j^N (z,\zeta) 
&= c_d r^{-N} |\zeta| 
\left[ (\pa_{z_j} r)  \left( \pa_r f(r,\thh) - \pa_r T^N_r f(r,\thh) \right) + (\pa_{z_j} \thh) \left( \pa_\thh f(r,\thh)  -  \pa_\thh T^N_r f(r,\thh)  \right) \right]  
\\ &= (2-d) r^{-N} |\zeta| (\pa_{z_j} r) \left\{  r [g(r,\thh) - T^{N-2}_r [g](r,\thh) ] - \cos\thh \cdot [g(r,\thh) - T^{N-1}_r [g](r,\thh)] \right\} 
\\ &\quad + (2-d) r^{-N} |\zeta| (\pa_{z_j} \thh) 
\left\{ (r \sin \thh) \left( g(r,\thh) - T^{N-1}_r [g](r,\thh) \right) \right\}. 
\end{aligned} 
\end{equation}

Next we compute $\pa_{z_j} r$ and $\pa_{z_j} \thh$. We have
\begin{gather*}
 \pa_{z_j} r = \pa_{z_j} \left( \frac{|z|}{|\zeta|} \right) = \yh \frac{\ov z_j}{|\zeta| |z|}. 
\end{gather*} 
To compute $\pa_{z_j} \thh$, we use the law of cosine
\[
  |z-\zeta|^2 = |z|^2 + |\zeta|^2 - 2 |z| |\zeta| \cos \thh. 
\]
Hence
\[
  \cos \thh = \frac{|z|^2 + |\zeta|^2 - |z-\zeta|^2}{2|z| |\zeta|} = \frac{z \cdot \ov \zeta + \zeta \cdot \ov z}{2|z| |\zeta|}, 
\]
where we used $|z-\zeta|^2 = |z|^2 + |\zeta|^2 - (z \cdot \ov \zeta + \zeta \cdot \ov z) $. Take $\pa_{z_j}$ on both sides of the first equality above to get
\begin{align*}
 (- \sin \thh) \pa_{z_j} \thh &= \yh \pa_{z_j}
 \left( \frac{|z|}{|\zeta|} + \frac{|\zeta|}{|z|} 
- \frac{|z-\zeta|^2}{|z||\zeta|}\right) 
= \yh \pa_{z_j} \left( r + \frac{1}{r} - \frac{|z-\zeta|^2}{|z| |\zeta| } \right)
\\ &= \yh \pa_{z_j} \left( r + \frac{1}{r} - r -\frac{1}{r} +\frac{z \cdot \ov \zeta + \zeta \cdot \ov z}{|z| |\zeta|} \right)
= \yh \pa_{z_j} \left( \frac{z \cdot \ov \zeta + \zeta \cdot \ov z}{|z| |\zeta|} \right).
\end{align*}
Further computation gives 
\begin{align*}
(r \sin \thh) (\pa_{z_j} \thh) 
= - \yh \frac{|z|}{|\zeta|} \left( \frac{\ov \zeta_j}{|z| |\zeta|} - \yh \frac{ \ov z_j}{|z|^3 |\zeta|} \left( z\cdot \ov \zeta + \ov z \cdot \zeta \right) \right) .   
\end{align*}
The last expression is in $O(|\zeta|^{-1})$ as $\zeta \to 0$. 
In view of \re{IN_inter} and the above computations, we can write 
\begin{align*}
  I_j^N(z,\zeta) 
&= c_d h_{j1}(z,\zeta) r^{-N} \left( g(r,\thh) - T^{N-1}_r [g](r,\thh) \right)
\\&\quad + c_d h_{j2}(z,\zeta) r^{-(N-1)} \left( g(r,\thh) - T^{N-2}_r [g](r,\thh) \right) 
\end{align*} 
where \begin{align*}
 h_{j1}(z,\zeta) &:= |\zeta| \left[ r \sin \thh 
(\pa_{z_j} \thh)  - \cos \thh (\pa_{z_j } r)\right]
\\ &= - \yh \frac{1}{|\zeta|} \left( \ov \zeta_j - \yh \frac{ \ov z_j}{|z|^2} \left( z\cdot \ov \zeta + \ov z \cdot \zeta \right) \right) - \frac{z \cdot \ov \zeta + \zeta \cdot \ov z }{2|z| |\zeta|} \cdot \frac{\ov z_j}{2 |z|}; 
\end{align*}
and
\begin{gather*}
  h_{j2}(z,\zeta) := |\zeta| \pa_{z_j } r =
\yh |\zeta| \frac{\ov{z_j}}{|\zeta| |z|} 
= \yh \frac{\ov z_j}{|z|}. \qedhere
\end{gather*} 
\end{proof} 

\begin{prop} \label{Prop::IN-1_estimate_1} 
Let $z, \zeta \neq 0$. There exists constant $C_{d,k}>0$ such that
 \begin{enumerate}[(i)]
 \item 
\[  
  \left| D_z^k \left[ I_j^N (z,\zeta) - |\zeta|^{d-1} \left( \frac{|\zeta|}{|z|} \right)^{N} B_j (z,\zeta) \right] \right| 
\leq C_{d,k} N^{d-1+k} |\zeta|^{-k}, \quad 
\text{when $|z-\zeta| < \frac{|\zeta|}{2N}$}. 
\] 
\item 
\[
  \left| D_z^k I_j^N (z,\zeta) \right| 
  \leq C_{d,k} N^{d-2+k} \min \{ N, |1-r|^{-1} \} |z|^{-k}, \quad 
  \text{when $|z-\zeta| > \frac{|\zeta|}{2N}$}. 
\]
\end{enumerate}
\end{prop} 
\begin{proof}
(i) We first reduce the proof to the case $|\zeta| =1$. 
Denote $\wti \zeta= \la \zeta$, $\wti z= \la z$, with $\la = |\zeta|^{-1}$. Then $|\wti \zeta| = 1$. Suppose the estimate holds for $\wti \zeta$ and $\wti z$:  
\[ 
\left| D_{\wti z}^k \left[ I_j^N (\wti z, \wti \zeta) - |\wti \zeta|^{d-1} \left( |\wti \zeta| |\wti z|^{-1}  \right)^N B_j (\wti z, \wti \zeta) \right] \right| \leq C_{d,k} N^{d-1+k}, \quad |\wti z - \wti \zeta| < \frac{|\wti \zeta|}{2N}.  
\]
The left-hand side is equal to 
\begin{align*}
  \left| D_{\wti z}^k \left[ |\wti \zeta|^{d-1} \left( |\wti \zeta| |\wti z|^{-1} \right)^N P_j^{N-1} (\wti z, \wti \zeta) \right] \right|    
&= \la^{d-1-k} \la^{-(d-1)} \ \left| D_z^k \left[ |\zeta|^d \left( |\zeta||z|^{-1} \right)^N P_j^{N-1} (z,\zeta)\right] \right| 
\\ &= \la^{-k} \left| D_z^k \left[ |\zeta|^d \left( |\zeta||z|^{-1} \right)^{N-1} P_j^{N-1} (z,\zeta)\right] \right|. 
\end{align*}
Thus 
\[
  \left| D_z^k \left[ I_j^{N-1} (\wti z, \wti \zeta) - |\wti \zeta|^d \left( |\wti \zeta| |\wti z|^{-1}  \right)^{N-1} B_j (\wti z, \wti \zeta) \right] \right| \leq C_{d,k} N^{d-1+k} \la^{k} = N^{d-1+k} |\zeta|^{-k}, 
\]
which proves the reduction.  

We show that for $|\zeta_0| = 1$,  
\[
  \left| D_z^k \left[ I_j^N - |z|^{-N}  B_j (z, \zeta_0) \right] \right| 
\leq C_{d,k}  N^{d-1+k} , \quad 
\text{when $|z-\zeta_0| < \frac{1}{2N}$}.
\]
Write 
\begin{equation} \label{der_diff_PN}  
  \left| D_z^k \left[ |z|^{N} \left( I_j^N - |z|^{-N}  B_j (z, \zeta_0) \right) \right] \right| 
  = |D_z^k P_j^{N-1}(z, \zeta_0) |. 
\end{equation}
By \rp{Prop::ul_sp_est}, we can expand the Newtonian potential using zonal harmonics:  
\[ 
 \frac{1}{|\zeta_0-z|^{d-2}}  = \sum_{l=0}^\infty Z_l(z), \quad |z| < 1, 
\]
where $Z_l$ satisfies the estimates $|Z_l(z)| \lesssim l^{d-3} |z|^l$. Since $P_j^{N-1}(\zeta_0, \cdot)$ is the degree $N-1$ Taylor polynomial of $B_j(\zeta_0,\cdot) = 
c_d \pa_{z_j} \frac{1}{|\zeta_0-z|^{d-2}}$ at $z =0$, we can write $P_j^{N-1}(\zeta_0, \cdot)$ as 
\[
  P_j^{N-1}(\zeta_0, \cdot) = T^{N-1} \left(\pa_{z_j} \frac{1}{|\zeta_0-z|^{d-2}} \right) 
  = \pa_{z_j} T^N \left(\frac{1}{|\zeta_0-z|^{d-2}} \right) 
  = \sum_{m=1}^{N} \pa_{z_j} Z_m(z). 
\]
Since each $Z_m$ is harmonic, we get (cf. \cite[Proprosition 1.13]{H-L_11} 
\[
  |D^k Z_m (z) | \leq \frac{C_{d,k}}{r^k} \sup_{B_r (z)} |Z_m|, \quad C_{d,k} = e^k e^{k-1} k!. 
\]
Choose $r= |z|/m \neq 0$. Then $|D^k Z_m | \leq C_{d,k} m^{d-3+k} |z|^{m-k}$. Thus 
\[
\left| D_z^k P_j^{N-1}(z, \zeta_0) \right| 
\leq \sum_{m=1}^N |D_z^{k+1} Z_m (z)| 
\leq C_{d,k} \sum_{m=1}^N m^{d-2+k} |z|^{m-k-1}  
  \leq N^k |z|^{-k} \sum_{m=1}^N m^{d-2} |z|^{m-1}. 
 \]
In view of \re{der_diff_PN} and the product rule, we have 
\begin{equation} \label{I-B_init_est} 
  \left| D_z^k \left[ I_j^N(z, \zeta_0) - |z|^{-N}  B_j (z, \zeta_0) \right] \right| \leq C_{d,k} N^k |z|^{-k -N } \sum_{m=1}^N m^{d-2} |z|^{m-1}.   
\end{equation}
For $|z-\zeta_0| < \frac{1}{2N}$, we have $|z|^{-k-N+m-1} \leq C_{d,k}$. Hence 
\begin{equation} \label{I-B_z~1_est}  
  \left| D_z^k \left[ I_j^N(z, \zeta_0) - |z|^{-N}  B_j (z, \zeta_0) \right] \right| 
 \leq C_{d,k} N^k \sum_{m=1}^N m^{d-2} \leq N^{d-1+k}.  
\end{equation}

The proof of (i) is complete. 
\\ 

Next we prove (ii). Again we first show that the proof can be reduced to the case $|\zeta|=1$. Suppose the estimate holds for $\wti \zeta = \la \zeta, \wti z = \la z$, for $\la = |\zeta|^{-1}$. For $|z-\zeta| \geq \frac{|\zeta|}{2N}$, we have $|\wti z - \wti \zeta| \geq \frac{|\wti \zeta|}{2N} = \frac{1}{2N}$, and 
\begin{equation} \label{I_N_scaled_est}  
    \left| D_{\wti z}^kI_j^N (\wti z, \wti \zeta) \right| 
  \leq C_{d,k} N^{d-2+k} \min \{ N, |1-r|^{-1} \} |\wti z|^{-k}.   
\end{equation}
We have 
\begin{align*}
I_j^N (\wti z, \wti \zeta) 
&= |\wti \zeta|^{d-1} \left( \frac{|\wti \zeta|}{|\wti z|} \right)^N \left( B_j (\wti z, \wti \zeta) - P_j^{N-1}(\wti z, \wti \zeta) \right) 
\\ &= \la^{(d-1)-(d-1)} |\zeta|^{d-1} \left( \frac{ |\zeta|}{|z|} \right)^{N} \left( B_j (z,\zeta) - P_j^{N-1} (z,\zeta) \right) 
\\ &=I_j^N (\zeta,z). 
\end{align*}
Thus
\[
 |D_{\wti z}^kI_j^N (\wti z, \wti \zeta) |
= \la^{-k} |D_z^k I_j^N (z,\zeta)|. 
\]
Together with \re{I_N_scaled_est}, we obtain 
\begin{align*}
 |D_z^kI_j^N (z,\zeta)|
&\leq C_{d,k} \la^{k} N^{d-2+k} \min \{ N, |1-r|^{-1} \} \la^{-k} |z|^{-k} 
\\ &= C_{d,k} N^{d-2+k} \min \{ N, |1-r|^{-1} \} |z|^{-k}, 
\end{align*}
which proves the reduction. 

We now fix $\zeta_0$ with $|\zeta_0| =1$ and show that 
\[
  |D_z^kI_j^N (z, \zeta_0) | 
  \leq C_{d,k} N^{d-2+k} \min \{ N, |1-r|^{-1} \} |z|^{-k}, \quad r = |z|, \quad |z-\zeta_0| > \frac{1}{2N}. 
\]
If $|z-\zeta_0| > \frac{1}{2N}$ and $|z| \in (1-\frac{1}{N}, 1+ \frac{1}{N})$, (i.e. $|1-|z||^{-1} > N$), the same argument as above shows that \re{I-B_z~1_est} holds: 
\begin{equation} \label{I-B_z~1_est_copy}   
   \left| D_z^k \left[ I_j^N(z, \zeta_0) - |z|^{-N}  B_j (z, \zeta_0) \right] \right| 
 \leq C_{d,k} N^{d-1+k} \leq C_{d,k} N^{d-1+k} |z|^{-k}.    
\end{equation} 
The last inequality holds since $|z| \lesssim  1$.   
Suppose $|z-\zeta_0| > \frac{1}{2N}$ and $|z| > 1- \frac{1}{N}$. We have 
\begin{align*}
  \left| D_z^k \left[ |z|^{-N}  B_j (z, \zeta_0) \right] \right| 
&= \left| \sum_{k_1+ k_2 =k} \na^{k_1}_z \left( |z|^{-N} \right)D_z^{k_2} \left( \frac{\ov{(\zeta_0)_j -z_j} }{|z-\zeta_0|^d} \right) \right|  
\\ &\leq C_{d,k} N^{k_1} |z|^{-N-k_1} |z-\zeta_0|^{-(d-1)-k_2}
\\ &\leq C_{d,k} N^{k_1} |z|^{-k_1} |z-\zeta_0|^{-(d-1)-k_2},  
 \end{align*}
 where in the last step we used $|z| > 1- \frac{1}{N}$ and thus $|z|^{-N} \leq (1-\frac{1}{N})^{-N}$.  
By the condition $|z-\zeta_0| > \frac{1}{2N}$, we have 
\[
 |z-\zeta_0|^{-(d-1)-k_2} 
 \leq C_{d,k} N^{d-1+k_2} |z|^{-k_2}. 
\] 
One can check the above inequality by considering the two cases: $|z| \geq 2 |\zeta_0| = 2$ and $ |z| \leq 2$.   
Therefore $ \left| D_z^k \left[ |z|^{-N}  B_j (z, \zeta_0) \right] \right| \leq C_{d,k} N^{d-1+k} |z|^{-k}$.   

On the other hand, since $|z-\zeta_0| \geq \left| 1 - |z| \right|$, by similar reasoning as above, 
\[
  |z-\zeta_0|^{-(d-1)-k_2} \leq C_{d,k}  N^{d-2+k_2}  |1-|z||^{-1} |z|^{-k_2}, 
\]
which implies that $ \left| D_z^k \left[ |z|^{-N}  B_j (z, \zeta_0) \right] \right| \leq C_{d,k} N^{d-2+k} |1-|z||^{-1} |z|^{-k}$. Hence 
\begin{equation} \label{z_N-1_B_est}  
   \left| D_z^k \left[ |z|^{-N} B_j (z, \zeta_0) \right] \right| \leq C_{d,k} 
  N^{d-2+k} \min \{N, |1-|z||^{-1} \} |z|^{-k} 
\end{equation}  when $ |z-\zeta_0| > \frac{1}{2N}$ and $|z| > 1- \frac{1}{N}$.

Combining \re{I-B_z~1_est_copy} and \re{z_N-1_B_est}, and using the fact that $N < |1-|z||^{-1}$ for $|z| \in 
(1-\frac{1}{N}, 1+ \frac{1}{N})$, we have  
\begin{equation} \label{IN-1_est_case1}  
|D_z^k I_j^N (z, \zeta_0) | 
\leq C_{d,k} N^{d-2+k} \min \{ N,|1-|z||^{-1} \}  |z|^{-k}, \: |z-\zeta_0|> \frac{1}{2N}, \:  |z| \in \left( 1-\frac{1}{N}, 1+ \frac{1}{N} \right).  
\end{equation}
When $|z| > 1+ \frac{1}{N}$, estimate \re{z_N-1_B_est} still holds. In this case we have $(|z|-1)^{-1} \leq N$, thus \re{z_N-1_B_est} reads  
\begin{equation} \label{z_N-1_B_est_case2}
   \left| D_z^k \left[ |z|^{-N}  B_j (z, \zeta_0) \right] \right| 
\leq C_{d,k} N^{d-2+k} (|z|-1)^{-1} |z|^{-k}.   
\end{equation}
From \re{I-B_init_est}, we get  
\begin{align*}
  \left| D_z^k \left[I_j^N(z, \zeta_0) - |z|^{-N}  B_j (z, \zeta_0) \right] \right| 
  &\leq C_{d,k} N^k |z|^{-k -N} \sum_{m=1}^N m^{d-2} |z|^{m-1}
\\ &\leq C_{d,k} N^{d-2+k} |z|^{-k-N} \sum_{m=1}^N |z|^{m-1}.
\end{align*}
By the geometric sum formula, we get 
$\sum_{m=1}^N |z|^{m-1} = (|z|-1)^{-1} (|z|^N-1)$, hence the last expression is bounded by
\[
  C_{d,k} N^{d-2+k} (|z| -1)^{-1} (|z|^N-1) |z|^{-k-N} \leq C_{d,k} N^{d-2+k} (|z| -1)^{-1} |z|^{-k}, \quad |z| > 1+\frac{1}{N}.  
\]
Thus $\left| D_z^k \left[ I_j^N(z, \zeta_0) - |z|^{-N}  B_j (z, \zeta_0) \right] \right| \leq C_{d,k} N^{d-2+k} (|z|-1)^{-1} |z|^{-k}$. 
Together with \re{z_N-1_B_est_case2}, we get 
\begin{equation} \label{IN-1_est_case2} 
|D_z^kI_j^N (z, \zeta_0) | \leq C_{d,k} N^{d-2+k} (|z|-1)^{-1} |z|^{-k}, \quad |z-\zeta_0|> \frac{1}{2N} \: \text{and} \:  |z| > 1 + \frac{1}{N}. 
\end{equation}
Finally, suppose $|z| < 1 - \frac{1}{N}$. We have
\begin{align*}
 |z|^{N}I_j^N(z, \zeta_0) 
  &= B_j (z,\zeta_0) - P_{N-1}^j(z,\zeta_0) 
 = c_d \pa_{z_j} \left( \frac{1}{|z-\zeta_0|^{d-2}} - \sum_{m \leq N} Z_m \right) 
 = c_d \pa_{z_j} \left( \sum_{m \geq N+1} Z_m \right).  
\end{align*}
 
By summing up a geometric series: 
\begin{align*}
  \left| D^k_z (|z|^{N} I^N_j(z,\zeta_0) ) \right| 
  &\leq C_{d,k}   \sum_{m \geq N+1} |D_z^{k+1} Z_m | 
  \leq C_{d,k} \sum_{m\geq N+1} m^{d-3+(k+1)} |z|^{m-k-1}
  \\ &\leq C_{d,k} N^{d-2+k} |z|^{N-k} (1-|z|)^{-1}.  
\end{align*}
Using the product rule, we have 
\[
  |D^k_z I_j^N(z, \zeta_0)| 
  \leq C_{d,k} N^{d-2+k} (1-|z|)^{-1} |z|^{-k} . 
\]
The proof is now complete.
\end{proof}
\begin{lemma} \label{Lem::triassic} 
Let $a \in \R$ and $\thh \in [0,2\pi]$. Then 
\[
  |a - e^{i \thh}| \gtrsim 
 |a-1| + |\sin \thh|.   
\]
\end{lemma}
\begin{proof}
The proof can be done by considering two cases.  
We have
\[
 |a - e^{i \thh}|^2 = (a- \cos \thh)^2 + \sin^2 \thh \geq \yh \left[ |a- \cos \thh)| + |\sin \thh| \right]^2.    
\]
Hence $|a - e^{i \thh}| \gtrsim  |a-\cos \thh| + |\sin \thh| $. 
If $|a- \cos \thh| > \yh |a-1|$, we are done. Suppose $|a- \cos \thh| \leq \yh |a-1|$. Then $|a - |\cos \thh| | \leq |a-\cos \thh| \leq \yh |a-1| $ and  
\begin{align*}
|\sin \thh| \geq \sin^2 \thh = 1 - \cos^2 \thh 
  \geq 1- |\cos \thh| 
  \geq |1-a| - |a-  |\cos \thh| | 
\geq \yh |a-1|, 
\end{align*}
which implies that $|a-e^{i \thh}| \gtrsim |\sin \thh| \gtrsim (|\sin \thh| + |a-1|)$. 
\end{proof} 
\begin{prop} \label{Prop::g-Tg}  
Let $ g(r,\thh) :=  \left[(r-e^{i\thh})(r-e^{-i \thh}) \right]^{-\frac{d}{2}}$. If $|\sin \thh| \gtrsim \frac{1}{N}$, then we can write 
\begin{equation} \label{g-Tg_a_form}   
 r^{-N} \left( g(r,\thh) - T^{N-1} [g](r,\thh) \right) = \RE \left[ a(r,\thh) e^{iN \thh} \right] 
\end{equation}
for suitable complex-valued function $a$ satisfying 
\begin{gather} \label{a_der_est}  
\left| \frac{\pa^i}{\pa r^i} \frac{\pa^k}{\pa \thh^k} a(r,\thh) \right| 
\leq C_{d,i,k} N^{\frac{d}{2}-1} |\sin \thh|^{-\frac{d}{2} -k} (|\sin \thh| + |1-r|)^{-1-i}. 
\end{gather}   

\end{prop}
\begin{proof}
Fix $\thh$, we complexify $g$ in the $r$ variable. 
The function $g^{\thh} (w):= [(w-e^{i\thh}) (w- e^{-i \thh})]^{-\frac{d}{2}}$ is analytic in $\C$ 
except for poles at $e^{\pm i \thh}$. 
For $w \in \C$, we denote $R_N (w):= g^\thh(w) - T_{N-1}(g)(w)$. 
By the Cauchy integral formula and the residue theorem, we have, for small $|w|$ and $N \geq 2$,  
\begin{align*}
 w^{-N} R_N(w)   
&= \frac{1}{2 \pi i} \int_{|\zeta| = 1/2} \frac{R_{N}(\zeta)}{\zeta^{N} (\zeta -w)} \, d \zeta 
 \\ &= \frac{1}{2 \pi i} \int_{|\zeta| = 1/2} \frac{g(\zeta)}{\zeta^{N} (\zeta -w)} \, d \zeta
 \\ &= - \left(  \underset{\zeta = e^{i\thh}}{\Res} + \underset{\zeta = e^{-i\thh}}{\Res} \right) \left[  \zeta^{-N} g(\zeta) (\zeta -w)^{-1} \right]
 \\ &\quad + \underset{L \to \infty}{\lim} \frac{1}{2 \pi i} \int_{|\zeta|=L}  \zeta^{-N} g(\zeta) (\zeta -w)^{-1} \, d \zeta. 
\end{align*}
The limit is $0$ for the last term, so 
\[
  w^{-N} R_N(w) = - \left(  \underset{\zeta = e^{i\thh}}{\Res} + \underset{\zeta = e^{-i\thh}}{\Res} \right) \left[  \zeta^{-N} (\zeta -w)^{-1} g(\zeta) \right].
\]
By analytic continuation the above holds for all $w$. 

Write
\[
 \zeta^{-N} (\zeta -w)^{-1} g(\zeta) 
 = \zeta^{-N} (\zeta -w)^{-1} \left[(\zeta-e^{i \thh})(\zeta-e^{-i \thh})\right] ^{-\frac{d}{2}}. 
\]  
The residues can be computed by taking $\eta = \frac{d}{2} -1$ derivatives: 
\begin{gather*}
  \underset{\zeta = e^{i\thh}}{\Res}  
 = \left. \frac{d^\eta}{d \zeta^\eta} \right|_{\zeta = e^{i \thh}} \zeta^{-N} (\zeta -w)^{-1} \left[(\zeta -e^{-i \thh})\right] ^{-\frac{d}{2}}  \\ 
  \underset{\zeta = e^{-i\thh}}{\Res}  
 = \left. \frac{d^\eta}{d \zeta^\eta} \right|_{\zeta = e^{-i \thh}} \zeta^{-N} (\zeta -w)^{-1} \left[(\zeta -e^{i \thh})\right]^{-\frac{d}{2}}. 
\end{gather*}
Hence we can write
\[
  r^{-N} R_N (r) = 
\RE \left( e^{i N \thh} a(r,\thh) \right) 
\]
where $a(r,\thh)$ is a finite sum of terms of the form 
\[
  a_{lpm} (e^{-i \thh} - r)^{-1-l} ( \sin \thh)^{-\frac{d}{2} - p},   
\]
$a_{lpm}$ being constants with $|a_{lpm}| \approx N^{m}$ and $l+p+m = d/2 -1$. Since $|\sin \theta| \gtrsim N^{-1}$ by assumption, and also $|r-e^{-i\thh}| \gtrsim |\sin \thh|$ 
(\rl{Lem::triassic}), the worst term is the term with $m = \frac{d}{2} -1$. Thus
\begin{equation} \label{a_z-zeta_bd}  
\begin{aligned}
 |a(r,\thh)| &\leq C_d N^{\frac{d}{2}-1} |\sin \thh|^{-\frac{d}{2}} |e^{-i\thh} - r|^{-1} 
 \\ &\leq C_d N^{\frac{d}{2}-1} |\sin \thh|^{-\frac{d}{2}} (|\sin \thh| + |1-r|)^{-1}.
\end{aligned}  
\end{equation} 
Similarly we can estimate $\frac{\pa^i}{\pa r^i} \frac{\pa^k}{\pa \thh^k} a(r,\thh) $. Each $r$ derivative produces a factor of $(e^{-i\thh} - r)^{-1}$, while each $\thh$ derivative produces a factor at worst a factor of $|\sin \thh|^{-1}$. Therefore 
\begin{align*}
\left| \frac{\pa^i}{\pa r^i} \frac{\pa^k}{\pa \thh^k} a(r,\thh) \right| 
&\leq C_{d,i,k} N^{\frac{d}{2}-1} |\sin \thh|^{-\frac{d}{2} -k} |e^{-i\thh} -r|^{-1-i} 
\\ &\leq C_{d,i,k} N^{\frac{d}{2}-1} |\sin \thh|^{-\frac{d}{2} -k} (|\sin \thh| + |1-r|) ^{-1-i}. \qedhere
\end{align*}
\end{proof}
The following result follows immediately from \rp{Prop::IN_hg_form} and \rp{Prop::g-Tg}. 
\begin{cor}
  Let $I_j^N, h_{1,j}, h_{2,j}$ be defined as above. For $z,\zeta \in \R^d \sm \{ 0 \}$ and $\thh(z,\zeta) \gtrsim N^{-1}$, we can write
\begin{equation} \label{IN_polar_form}  
   I_j^N (z,\zeta) = h_{1,j}(z) \RE \left[ a_1(r,\thh) e^{iN\thh} \right] + h_{2,j} (z,\zeta) \RE \left[ a_2(r,\thh) e^{i(N-1)\thh} \right].    
\end{equation}
Furthermore, let $b(z,\zeta)$ denote one of the following: $h_{1,j}(z) a_1 (r,\thh)$, $h_{1,j}(z) \ov{a_1 (r,\thh)}$, $h_{2,j} (z,\zeta) a_2 (r,\thh)$ and $h_{2,j} (z,\zeta) \ov{a_2 (r,\thh)}$. Choose a coordinate system on the unit sphere $S^{d-1}$. Let $\xi$ and $\eta$ be variables on $S^{d-1}$, $D_\xi^\si$, $D_\eta^\tau$ denote differentiation in the given coordinate system. Then for $| \sin \thh | > \frac{1}{2N}$, 
\begin{equation} \label{bi_sph_est}  
  \left| D_\xi^\all D_\eta^\beta b_i(z,\zeta) 
\right| \leq C_{d,\all,\beta} N^{\frac{d}{2}-1} |\sin \thh|^{-\frac{d}{2} -|\all|-|\beta|} (|\sin \thh| + |1-r|)^{-1}, \quad i =1,2.   
\end{equation}
\end{cor}
\begin{proof}
  The form of $I_j^N$ \re{IN_polar_form} is immediate from \re{IN_hg_form} and \re{g-Tg_a_form}. To show \re{bi_sph_est}, recall that 
  \[
    h_{1,j}(z,\zeta) = - \frac{1}{|\zeta|} \left( \ov \zeta_j - \yh \frac{ \ov z_j}{|z|^2} \left( z\cdot \ov \zeta + \ov z \cdot \zeta \right) \right) - \frac{z \cdot \ov \zeta + \zeta \cdot \ov z}{2|z| |\zeta|} \cdot \frac{\ov z_j}{2 |z|}; \quad 
  h_{2,j}(z) = \frac{\ov z_j}{2 |z|}. 
  \]
  Writing in polar coordinates: $z= s w(\xi) = s(w_1(\xi), \dots, w_n(\xi))$, $\zeta = t w(\eta) = t(w_1(\xi), \dots, w_n(\xi))$, where $s =|z|, 
t=|\zeta|$, we have
\begin{align*}
  h_{1,j}(s \xi, t\zeta) 
  &= - \left( \ov{w_j(\eta)} - \yh \ov {w_j(\xi)} \left( w(\xi) \cdot \ov{w(\eta)} + w(\xi) \cdot \ov{w(\eta)}   \right) \right) - \yh \left[ w(\xi) \cdot \ov{w(\zeta)} \right] \ov{w_j(\xi)}  
  \\ &\quad - \yh \left[ \ov{w(\xi)} \cdot w(\zeta) \right] \ov{w_j(\xi)} 
\end{align*}
and $h_{2,j}(s \xi, t\zeta) = \yh \ov{w_j(\xi)}$.  
Then the estimate follows by \re{a_der_est} and the fact that 
 \[
  |D_\xi^\all  w(\xi)  | \leq 1,  \quad 
|D_\eta^\beta w(\eta)| \leq 1, \quad  \forall \: \all, \beta. \qedhere
 \]
\end{proof}

\begin{prop} \label{Prop::I_N-1_est} 
  Let $u \in C^\infty_c (\C^n \sm \{ 0 \})$. We have
\[
  |z|^{-N} u(z)=  \int_\Om I^N(z,\zeta) |\zeta|^{-(N+d-1)} \we \db u(\zeta).
\]
Furthermore, 
\\ \\ 
\medskip
(i) 
  If $|\zeta-z| < \frac{|\zeta|}{2N}$, then 
\[
  |I^N(z,\zeta) | \leq C_d |\zeta|^{d-1} |\zeta-z|^{-(d-1)}. 
\]
(ii) 
If $|\zeta-z| \geq \frac{|\zeta|}{2N}$, then 
\[
   |I^N (z,\zeta) | \leq C_d   
N^{d-2} \min \{N, |1-r|^{-1} \}.   
\]

\end{prop}

\begin{proof}
(i) We have $I^N(z,\zeta) = c_d \sum_{j=1}^n I_j^N(z,\zeta) \hht{d\ov \zeta_j} $. For each $I_j^N$ we have 
\begin{equation} \label{I_N-1_est_tr_ineq}   
   |I_j^N (\zeta,z)| 
\leq \left| I_j^N(z,\zeta) - |\zeta|^{d-1} \left( \frac{|\zeta|}{|z|} \right)^N B_j (z,\zeta) \right| + |\zeta|^{d-1} \left( \frac{|\zeta|}{|z|} \right)^N \left| B_j (z,\zeta) \right|.    
\end{equation}
By \rp{Prop::IN-1_estimate_1} (i), the first term on the right is bounded by $C_d N^{d-1}$. The assumption $N \lesssim \frac{|\zeta|}{|\zeta-z|}$ implies that
\[
  N^{d-1} |\zeta| \lesssim |\zeta|^{d-1} |\zeta-z|^{-(d-1)}. 
\]
To estimate the last term in \re{I_N-1_est_tr_ineq}, notice that 
\[
  |\zeta-z| < \frac{|\zeta|}{2N} \iff 
  |\wti \zeta - \wti z| < \frac{1}{2N}, \quad \wti \zeta = \la \zeta, \; \wti z = \la z, \; \la = |\zeta|^{-1}.  
\]
Since $|\wti \zeta| =1 $, the above condition implies that 
$|\wti z|^{-1} \leq \frac{1}{1-(2N)^{-1}} = \frac{2N}{2N-1} $. Thus
\[
  |\wti z|^{-N} \leq \frac{(2N)^N}{(2N-1)^N} = \left( \frac{2N-1}{2N} \right)^N = \left[ \left(1- \frac{1}{2N} \right)^N \right]^{-1} \leq C.   
\] 
\begin{align*}
  \left( \frac{|\zeta|}{|z|} \right)^N |\zeta|^{d-1}  \left| B_j (z,\zeta) \right| 
&=  \left( \frac{|\wti \zeta|}{|\wti z|} \right)^N |\zeta|^{d-1}  \left| B_j (z,\zeta) \right| 
\\ &\lesssim  \la^{-(d-1)} \la^{d-1} \left| B_j (\wti z, \wti \zeta) \right| 
\\ &\lesssim |\wti \zeta - \wti z|^{-(d-1)} 
\\ &= \la^{-(d-1)} |\zeta-z|^{-(d-1)} 
= |\zeta|^{d-1} |\zeta-z|^{-(d-1)}. 
\end{align*}
\\ \medskip 
\noindent 
(ii) 
The estimate follows immediately from \rp{Prop::IN-1_estimate_1} (ii). 
\end{proof}

\section{Oscillatory integral on the sphere}
This section largely follows Wolff \cite[Section 2]{Tom90}, and we shall adopt most of his notations.
\\ 
\nid
\textbf{Notations.}  
If $K$ is a function on $\R^d \times \R^d$, $s,t>0$, then $K^{st}$ is the function on $S^{d-1} \times S^{d-1}$ defined by $K^{st}(e,f) = K(se,tf)$. 

We fix a coordinate system on $S^{d-1}$. For $\la \in (0,1)$, denote by $\chi_\la$ any function from $S^{d-1} \times S^{d-1}$ to $[0,1]$ satisfying the following: $\chi_\la(x,y) = 0$ if $|\sin (\thh) (\xi,\eta)| < \frac{\la}{100}$ or $|\sin (\thh) (\xi,\eta)| > 100 \la$, and $|D^\all \chi_\la| \leq 
C_{|\all|} \la^{-|\all|}$. We denote by $\chi^\la$ any function from $S^{d-1} \times S^{d-1}$ to $[0,1]$ satisfying $\chi^\la(x,y)= 0$ if $|\sin \thh (x,y)| > 100 \la$, and $|D^\all \chi^\la| \leq C_{|\all|} \la^{-|\all|} $. That is, $\chi_\la$ is a smooth cutoff to $\{ \sin \thh = \la \}$ and $\chi^\la$ is a smooth cutoff to $|\sin \thh| \lesssim \la$.  

We shall write $A \lesssim B$ if $A \leq C B$ where $C$ is a constant which can depend on $d$ and $p$ but is independent of $N$, $\nu$ and $\la$. We write $A \approx B$ if both $A \lesssim B$ and $B \lesssim A$. 

We will identify kernel with the operator it induces, and denote the norm of an operator acting from $L^p(X,\mu) \to L^q(Y,\nu)$ by $\|T \|_{L^p(X,\mu) \to L^q(Y,\nu)}$, or $\| T \|_{p\to q}$ if there is no confusion.   

\begin{lemma} \label{Lem::op_norm_prod_space} 
 Suppose $1 < p < q <\infty$, $(X,\mu), (Y, \nu), (Z, \si), (W,\tau)$ are measure spaces, $\{ T_{wy} \}_{y \in Y, w \in W}$ is a measurable family of operators from $L^p (X,\mu)$ to $L^q (Z, \nu)$ and the kernel 
\[
  n(w,y) = \| T_{wy} \|_{L^p(X,\mu) \to L^q (Z,\si)} 
\]
defines a bounded operator $f \to \int n(w,y) f(y) \, d \nu(y)$ from $L^p (Y,\nu)$ to $L^q (W,\tau )$ with norm $N$. For $f: X \times Y \to \C$ define $f_y(x) = f(x,y)$. Then $T$ defined by 
\[
  (Tf)_w = \int T_{wy} (f_y) \, d \nu(y) 
\]
is a bounded operator from $L^p(X \times Y, \mu  \times \nu ) \to L^q (Z \times W, \si \times \tau)$ with norm $\leq N$.   
\end{lemma}
\begin{proof}
  See \cite[Lemma 2.1]{Tom90}. 
\end{proof} 
\begin{lemma} \label{Lem::op_norm_AB} 
   Suppose $1 \leq p \leq 2$, $\displaystyle \frac{1}{p}- \frac{1}{p'} = \frac{1}{q}$, $(X,\mu)$ and $(Y,\nu)$ are measure spaces, $K:X \times Y 
 \to \C$ and $u: X \to \R^+$, $v: Y \to \R^+$. Define 
\begin{gather*}
A = \sup_x \|(u(x) v(y))^{-\frac{1}{p'}} K(x,y) \|_{L^{q'}(Y,v(y)\nu)} ; 
   \\ 
 B = \sup_y \|(u(x) v(y))^{-\frac{1}{p'}} K(x,y) \|_{L^{q'}(Y,u(x) \mu)}. 
\end{gather*}
Then the norm of $f \to \int K(x,y) f(y) \, d \nu(y)$ as an operator from $L^p(Y,\nu)$ to $L^{p'} (X,\mu)$ is $\leq$ $(AB)^{1/2} $. 
\end{lemma}
\begin{proof}
  See \cite[Lemma 2.2]{Tom90}. 
\end{proof} 

\begin{lemma} \label{Lem::oscill_est}   
  Suppose $\thh: \R^{d-1} \times \R^{d-1} \to \R $, $a \in C^\infty_c (\R^{d-1} \times \R^{d-1})$. Assume $|x-y| \approx \rho$ for all $x,y \in \supp a$, and on $\supp a$, we have $|D_x^\all D_y^\beta \thh| \lesssim C_{\all\beta} \rho^{1-|\all|-|\beta|} $, 
  and $\na_x \na_y \thh$ has at least $m$ eigenvalues with magnitude $\geq (C_1 \rho)^{-1}$. Assume moreover that 
 $\left| D_x^\all D_y^\beta a \right| \leq C_{\all \beta} \rho^{-|\all| -|\beta|}$. Then $a(x,y) e^{iN \thh(x,y)}$ is $L^p \to L^{p'} $ bounded with norm $\leq C \rho^{\frac{2(d-1)}{p'}} (N \rho)^{-\frac{m}{p'}} $, $1 \leq p \leq 2 $, where the constant $C$ depends on $d,C_1,C_{\all \beta}$. 
\end{lemma}
\begin{proof}
  See \cite[Lemma 2.4]{Tom90}.  
\end{proof} 

\begin{prop} \label{Prop::K^st_N-1_est} 
 Let $\psi$ be a smooth function on $\R^d \times \R^d$ such that 
 \[ 
\psi(z,\zeta) = 
\begin{cases}
 1 & \text{if $|z-\zeta| > \frac{1}{N} |\zeta|$}; 
 \\ 
 0 & \text{if $|z-\zeta| < \frac{1}{2N} |\zeta|$}, 
\end{cases}
\]   
and $\left| D^k \psi \right| \lesssim \left( \frac{|\zeta|}{N} \right)^{-k}$. Let $I^N$ be the kernel in \rp{Prop::IN_hg_form} and $K_N = \psi I^N$. Then for $1 \leq p \leq 2$, $\displaystyle \frac{1}{p} - \frac{1}{p'} = \frac{1}{q}$, the following hold: 
\begin{enumerate}[(i)] 
  \item 
Fix $\mu \geq 1$. If $ \frac{1}{200} N^{-\frac{1}{1+\mu}} \leq \la \leq 200 N^{-\frac{1}{1+\mu}}$, then $\chi^\la K_N^{s,t}$ maps $L^p(S^{d-1}) \to L^{p'} (S^{d-1})$ with norm bounded by  
\[
 C_{d,p} N^{\frac{d}{2 \mu q} - \frac{1}{\mu q}} \la^{-\frac{d}{2q} + d-1 - \mu (d-2)} \left( \left| 1- \frac{s}{t} \right| + \la^\mu  \right)^{-1}. 
\]
\item 
If $ \la \geq \frac{1}{200} N^{-1}$, then $\chi_\la K_N^{s,t}$ maps $L^p(S^{d-1}) \to L^{p'} (S^{d-1})$ with norm bounded by 
\[
  C_{d,p} N^{\frac{d}{2q} - \frac{1}{q}} \la^{-\frac{d}{2q} } \left( \left| 1- \frac{s}{t} \right| + \la \right)^{-1}. 
\]
\end{enumerate}
\end{prop} 
\begin{proof} 
As before we denote $r=\frac{s}{t}$. By \rp{Prop::I_N-1_est} (ii), we have 
\begin{align*}
 \| \chi^\la K^{s,t}_N \|_{\infty} 
&= \| \chi^\la \psi I^N (s \, \cdot, t \, \cdot) \|_{L^\infty (S^{d-1} \times S^{d-1})} 
\\ &\leq N^{d-2} \min \{ N, |1-r|^{-1} \} 
\\ &\leq N^{d-2} \left( \la^\mu + |1-r|\right)^{-1}.  
\end{align*}
To see the last inequality, we consider two cases. First, suppose $N < |1-r|^{-1}$. Since $\la \approx N^{-\frac{1}{\mu}}$, we have $\la^{-\mu} \approx N$ and therefore $\la^{-\mu} \lesssim |1-r|^{-1}$ or $\la^\mu \gtrsim |1-r|$. Thus we have 
\[
 N \approx \la^{-\mu} \lesssim (2 \la^\mu)^{-1} 
 \lesssim (\la^\mu + |1-r|)^{-1}. 
\]
On the other hand if $N > |1-r|^{-1}$, then $\la^{-\mu} \gtrsim |1-r|^{-1}$ or $\la^\mu \lesssim |1-r|$. It follows that 
\[
  |1-r|^{-1} \lesssim (2|1-r|)^{-1} 
  \lesssim (\la^\mu + |1-r|)^{-1}. 
\]

For fixed $\xi$, the set $\{ \eta \in S^{d-1}: \chi^\la K_N (\xi, \eta) \neq 0 \}$ has measure $\approx$ $N^{-\frac{d-1}{\mu}}$, since $\supp 
\chi^\la \subset \{ |\sin \thh (\xi, \eta)| \lesssim N^{-\frac{1}{\mu}} \}$. Hence 
\begin{align*}
 \sup_{\xi \in S^{d-1} } 
\| \chi^\la K^{s,t}_N (\xi, \cdot) \|_{L^{q'} ( d \eta)} 
&= \sup_{\xi \in S^{d-1} }  
\left( \int_{S^{d-1}} \left| \chi^\la K^{s,t}_N (\xi, \eta)  \right|^{q'} \, d \eta \right)^{\frac{1}{q'}} 
\\ &\lesssim N^{d-2 - \frac{d-1}{\mu q'}}  \left( \la^\mu + |1- r|  \right)^{-1}.  
\end{align*}

Similarly, we have 
\[
  \sup_{\eta \in S^{d-1} } 
\| \chi^\la K^{s,t}_N (\cdot, \eta) \|_{L^{q'} ( d \xi)} 
\lesssim  N^{d-2 - \frac{d-1}{\mu q'}}  \left( \la^\mu + |1- r| \right)^{-1}. 
\] 
By \rl{Lem::op_norm_AB}, we get 
\begin{align*}
 \| \chi^\la K^{s,t}_N \|_{p \rightarrow p'} 
&\lesssim N^{d-2 - \frac{d-1}{\mu q'}}  \left(\la^\mu + |1- r| \right)^{-1}. 
\end{align*}
Since 
\begin{align*}
d-2 - \frac{d-1}{\mu q'} &=  d-2 - \frac{d-1}{\mu } \left(1-\frac{1}{q} \right)
\\ &= \frac{d-1}{\mu q} + \left( d-2 - \frac{d-1}{\mu } \right)
\\ &= \left( \frac{d}{2 \mu q} - \frac{1}{\mu q} \right) - \left( - \frac{d}{2\mu q} + \frac{d-1}{\mu } - (d-2) \right) ,
\end{align*}
we have 
\[
  N^{d-2-\frac{d-1}{\mu q'}} = N^{\frac{d}{2 \mu q} - \frac{1}{\mu q}} (N^{-1})^{ - \frac{d}{2 \mu q} + \frac{d-1}{\mu } - (d-2) }. 
\]
Now $N^{-1} \approx \la^\mu $, so we have 
\[
   N^{d-2-\frac{d-1}{\mu q'}} \approx N^{\frac{d}{2\mu q} - \frac{1}{\mu q}} \la^{-\frac{d}{2q} + d-1 - \mu (d-2)} . 
\]
Hence 
\[ 
   \| \chi^\la K^{s,t}_N \|_{p \rightarrow p'} 
\lesssim N^{\frac{d}{2 \mu q} - \frac{1}{\mu q}} \la^{-\frac{d}{2q} + d-1 - \mu (d-2)} \left( \la^\mu + |1- r| \right)^{-1} . 
\]
\medskip 

\noindent 
(ii) Assuming that $|\sin \thh| \geq \frac{1}{200} N^{-1}$, we can write $I_j^N(z,\zeta)$ as a linear combination of 
\[
  b(z,\zeta) e^{iN\thh}, \quad b(z,\zeta) e^{-iN\thh},   
\] 
where $b(z,\zeta)$ is some function satisfying (\re{a_z-zeta_bd}):  
\[
  \left| D^\all_\xi D^\beta_\eta 
 \left[ b (s \xi, t\eta) 
\right] \right| \leq C_{d,\all,\beta} N^{\frac{d}{2} -1} |\sin \thh|^{-\frac{d}{2}-|\all|-|\beta|} (|\sin \thh| + |1-r|)^{-1}. 
\] 
Set 
\[
  \wti b(\xi,\eta) = N^{1-\frac{d}{2}} \la^{\frac{d}{2}} \left( \la + \left| 1- r \right| \right) b(s \xi, t \eta), \quad r =\frac{s}{t}. 
\] 
Since $\supp \chi_\la \subset \{ \sin \thh \approx \la \} $, we have  
\[
 |D^\all_\xi D^\beta_\eta (\chi_\la \wti b) | 
\leq C_{d,\all,\beta} \la^{-|\all|-|\beta| } . 
\]
Next, we observe that the phase function $\thh: S^{d-1} \times S^{d-1}$ has the property that $\na_x \na_y \thh$ (relative to a coordinate system) has $d-2$ eigenvalues with magnitude $\approx |\sin \thh|^{-1}$.  
Now, we can apply \rl{Lem::oscill_est} with $\rho= c \la $ and $m=d-2$. It follows that 
$\chi_\la \wti b e^{ \pm i N \thh}$ is $L^p(S^{d- 
1}) \to L^{p'}(S^{d-1})$ bounded with norm $\lesssim                                       \la^{\frac{2(d-1)}{p'}} (N\la)^{-\frac{d-2}{p'}}  $, and 
\begin{align*}
 \| \chi_\la K^{s,t}_N \|_{p \to p'} 
&\lesssim N^{\frac{d}{2}-1} \la^{-\frac{d}{2}} 
\left( \la + \left| 1- \frac{s}{t} \right| \right)^{-1} \la^{\frac{2(d-1)}{p'}} (N\la)^{-\frac{d-2}{p'}} 
\\&= N^{\frac{d-2}{2r}} \la^{-\frac{d}{2q}} \left( \la + \left| 1- \frac{s}{t} \right| \right)^{-1}. \qedhere
\end{align*}
 \end{proof}

\begin{prop}
 For sufficiently large $\nu \in \R^+$, there is a kernel $L_\nu$ such that 
\[
  |z|^{-\nu} u(z) = \int_{\C^n} L_\nu(z,\zeta) |\zeta|^{-\nu} \we \db u(\zeta), \quad u\in C^\infty_c(\C^n \sm \{0 \} )  
\] 
and $L_\nu = M_\nu + N_\nu$ where for suitable $C$, 
\begin{enumerate}[(i)]
  \item 
  Suppose $1 < p, q < \infty$ with $\displaystyle \frac{1}{p} - \frac{1}{q} = \frac{1}{m}$, for some $m > d$. Then $M_\nu$ maps $L^p(B(0,1))$ to $L^q (\R^d)$ with norm bounded independently of $\nu$. 
  \item 
  $N_\nu (z,\zeta) = 0$ if $|z - \zeta| < (200 
\nu)^{-1} |\zeta|$, and for any $\la > (200 \nu)^{-1}$, 
\begin{equation} \label{chi_la_Nst_nu_est}  
\| \chi_\la N_\nu^{s,t} \|_{p \to p'} 
\leq C_{d,p} 
  \nu^{\frac{d}{2q} - \frac{1}{q}} \la^{-\frac{d}{2q} } (st)^{-\frac{d-1}{2}} W_{s,t,\la} 
\end{equation} 
where 
\begin{align*}
 W_{s,t,\la}    
\leq \begin{cases} \displaystyle
    \left( \frac{s}{t} \right)^\rho, & \text{if $\displaystyle s < \frac{t}{2}$}, 
\\[8pt] \displaystyle
   \left( \left| 1- \frac{s}{t} \right| + \la \right)^{-1}, & \text{if $ \displaystyle \frac{t}{2} < s < 2t$},  
\\[8pt] \displaystyle
    \left( \frac{t}{s} \right)^{1-\rho}, & \text{if $ \displaystyle s > 2t$}.
\end{cases}
\end{align*}  

Here $\rho$ is defined to be the number in $[0,1)$ such that $\nu - \frac{d-1}{2} + \rho$ is an integer. 
\item If $ \displaystyle \frac{1}{200} \frac{1}{\nu^\frac{1}{\mu}} < \la < \frac{200}{\nu^\frac{1}{\mu}}$, 
then the following estimate holds: 
\begin{equation} \label{chi^la_Nst^nu_est}  
\| \chi^\la N_\nu^{s,t} \|_{p \to p'} 
\leq C_{d,p}  
  \nu^{\frac{d}{2 \mu q} - \frac{1}{\mu q}} \la^{-\frac{d}{2q} + (d-1) - \mu(d-2) } (st)^{-\frac{d-1}{2}} W_{s,t,\la^\mu}.  
\end{equation} 
\end{enumerate}
\end{prop}
\begin{proof}
  Write $\nu = N + \frac{d-1}{2} - \rho$, where $N \in \Z$, and $\rho$ is as above. Define 
\[
  L_\nu(z,\zeta) = |z|^{-\frac{d-1}{2} + \rho } |\zeta|^{-\frac{d-1}{2}-\rho} I^N (z,\zeta). 
\]
Since $ -\frac{d-1}{2} - \rho - \nu = -(d-1) -N $, we have 
\begin{align*}
 \int L_\nu (z,\zeta) |\zeta|^{-\nu} \we \db u(\zeta) 
&= |z|^{-\frac{d-1}{2} + \rho } \int I^N (z,\zeta) |\zeta|^{-(N+d-1)} \we \db u(\zeta)
\\ &= |z|^{-\frac{d-1}{2} + \rho} |z|^{-N} u(z) = |z|^{-\nu} u(z). 
\end{align*}
Define $M_\nu = (1-\psi) L_\nu $, $N_\nu = \psi L_\nu$. By \rp{Prop::I_N-1_est} (i), we have 
\[
  | I^N(z,\zeta) | \lesssim |\zeta|^{d-1} |\zeta -z|^{-(d-1)}, \quad \text{if $|\zeta-z| \lesssim \frac{|\zeta|}{N} $}. 
\]
This implies that 
\begin{align*}
 |M_\nu(z,\zeta)|
\lesssim (|z| |\zeta| ) ^{-\frac{d-1}{2}} \left( \frac{|z|}{|\zeta|} \right)^\rho |\zeta|^{d-1} |\zeta -z|^{-(d-1)}, 
\end{align*}
and $\supp M_\nu \subset \{ |\zeta -z| < N^{-1}|\zeta| \}$.  
For $|\zeta-z| \lesssim \frac{|\zeta|}{N}$, the triangle inequalities give $|z| \approx |\zeta|$ for all $N$ sufficiently large. Thus 
\[
  |M_\nu(z,\zeta)| \lesssim |\zeta|^{-(d-1)} |\zeta|^{d-1}  |\zeta -z|^{-(d-1)}
  = |\zeta -z|^{-(d-1)}.
\]
By \rl{Lem::op_norm_AB} it suffices to show that 
\[
  \| M_\nu (z,\zeta) \|_{L^{m'}(B(0,1))} \leq C, \quad m' = 
\frac{m}{m-1},  
\]
for some constant $C$ independent of $z$ and $\zeta$ and $\nu$. Indeed, since $ m >d$, $m' < d'= \frac{d}{d-1}$. Hence
\begin{align*}
 \int_{B(0,1)} |M_\nu(z,\zeta)|^{m'} d V(\zeta) \lesssim \int_{B(0,1)} |\zeta-z|^{-m(d-1)} d V(\zeta) \leq C.  
\end{align*}
\medskip

\noindent (ii) 
By \rp{Prop::K^st_N-1_est} (ii), for $\la > (200 \nu)^{-1}$, 
\begin{align*}
\| \chi_\la N^{s,t}_{\nu} \|_{p \to p'} 
&= \| \chi_\la \psi L^{s,t}_\nu \|_{p \to p'}  
\\ &= \left( \frac{s}{t} \right)^\rho (st)^{-\frac{d-1}{2}} \| \chi_\la K^{s,t}_N \|_{p \to p'} 
\\ &\lesssim N^{\frac{d}{2q} - \frac{1}{q}} \la^{-\frac{d}{2q}}  \left( \frac{s}{t} \right)^\rho (st)^{-\frac{d-1}{2}} \left( \left| 1- \frac{s}{t} \right| + \la \right)^{-1}.
\end{align*} 
Denote $W_{s,t,\la} :=  \left( \frac{s}{t} \right)^\rho \left( \left| 1- \frac{s}{t} \right| + \la \right)^{-1}$. The following is checked easily. 
\[
W_{s,t,\la} 
\lesssim 
\begin{cases} \displaystyle
    \left( \frac{s}{t} \right)^\rho, & \text{if $\displaystyle s < \frac{t}{2}$}, 
\\[8pt] \displaystyle
   \left( \left| 1- \frac{s}{t} \right| + \la \right)^{-1} & \text{if $ \displaystyle \frac{t}{2} < s < 2t$}, 
\\[8pt] \displaystyle 
    \left( \frac{t}{s} \right)^{1-\rho}, & \text{if $ \displaystyle s > 2t$}.
\end{cases}
\]
\medskip
(iii) By \rp{Prop::K^st_N-1_est} (i), if $\la \approx \nu^{-\frac{1}{\mu}}$, we have 
\begin{align*}
 \| \chi^\la N^{s,t}_\nu \|_{p \rightarrow p'} 
&= \| \chi^\la \psi L_\nu^{s,t} \| 
\\ &= \left( \frac{s}{t} \right)^\rho (st)^{-\frac{d-1}{2}} \| \chi^\la K^{s,t}_N \|_{p \to p'} 
\\ &\lesssim N^{\frac{d}{2 \mu q} - \frac{1}{\mu q}} \la^{-\frac{d}{2q} + d-1 - \mu(d-2)} (st)^{-\frac{d-1}{2}} W_{s,t,\la^\mu}. 
\qedhere\end{align*}
\end{proof}

\section{Carleman Inequalities} 
Fix $p$ and $q$ with $ \displaystyle \frac{1}{p} - \frac{1}{p'} = \frac{1}{q}$. For $z,\zeta \in \C^n$ we let $s = |z|$, $t= |\zeta|$, $\si = \log (1/s)$, $\tau 
= \log (1/t)$. If $\gm \subset \R$, then $\gm_\ast = \{ s \in \R: \log (1/s) \in \gm \}$, $A(\gm) = \{ x \in \R^d: |x| \in \gm_\ast \}$. We write $|\gm|' = \min \{\gm, \nu^{-\yh} \}$. 

Let $\psi: \R^+ \to \R$ be increasing and convex. We follow Wolff and osculate the weight $|x|^{-N}$ to $e^{\nu \psi(\tau)}$.
\begin{prop}
The following formula holds 
\begin{equation} \label{Pnu_formula_CCinfty} 
 e^{\nu \psi (\si)} u (z) = \int_{\C^n} P_\nu (z,\zeta) e^{\nu \psi(\tau)} \we \db u(\zeta), \quad u \in C^\infty_c(\C^n \sm \{ 0 \}), 
\end{equation}
where $P_\nu (z,\zeta) := e^{-\nu(\psi(\tau) - \psi(\si) 
- \psi'(\si)(\tau -\si))} L_{\nu \psi'(\si)} (z,\zeta) $. 
\end{prop}
\begin{proof}
Since $e^{\nu \psi'(\si)\si} = |z|^{-\nu \psi'(\si)} $, we get    
\begin{align*}
  e^{\nu \psi(\si)} u(z) &= e^{\nu (\psi(\si)- \psi'(\si) \si)} |z|^{-\nu \psi'(\si)} u(z) 
  \\&= \int e^{\nu (\psi(\si)- \psi'(\si) \si)}  L_{\nu \psi'(\si)} (z,\zeta) |\zeta|^{-\nu \psi'(\si)} \we \db u (\zeta) 
  \\ &= \int e^{\nu (\psi(\si)- \psi'(\si) \si)} 
  L_{\nu \psi'(\si)} (z,\zeta) e^{\nu (\psi'(\si) \tau-\psi(\tau))} \we  e^{\nu \psi(\tau)} \db u(\zeta) 
  \\ &= \int e^{\nu(\psi(\si)-\psi(\tau) + \psi'(\si)(\tau - \si))}  L_{\nu \psi'(\si)} \we  e^{\nu \psi(\tau)}\db u(\zeta). \qedhere
\end{align*}
\end{proof}
\begin{lemma} \label{Lem::min_est}
  Let $\la >0$, $\rho \geq 0$, and $q' \in (1,\infty)$. Then for intervals $\gm \subset \R$, there exists a constant $C_q$ depending only on $q$ such that 
\[
  \| e^{-\rho x^2} (|x|+\la)^{-1} \|_{L^{q'} (\gm)} \leq C_{q'} a^{-1} \min \{ \la, |\gm|, \rho^{-\yh} \}^{1/q'}.   
\]
\end{lemma}
\begin{proof}
    This is stated in \cite{Tom90} and we will provide details. There are three cases according to the relative size of $\la, |\gm|, \rho^{-\yh} $. In what follows we denote 
    \[
    I = \int_{\gm} e^{\rho q'x^2} (|x|+\la)^{-q'} \, dx. 
    \] 
We assume that $|x|<a$ for all $x \in \gm$. 
\textit{Case 1: $\rho^{-\yh} = \min \{ \la, |\gm|, \rho^{-\yh} \} $.} 
\begin{align*}
    I \leq \la^{-q'} \int_{\gm} e^{-\rho q' x^2} \, dx &= \la^{-q'} \left( \int_{\gm \times \gm} e^{-\rho q'(x^2+y^2)} \, dx \, dy \right)^{\yh} 
\\ &\leq \la^{-q'} \left( \int_{B(0,\sqrt 2 a ) } e^{-\rho q' r^2} \, r \, dr \, d \thh \right)^\yh  
\approx \la^{-q'} \left( \int_0^{\sqrt 2 a} e^{-\rho q' u} \, du \right)^{\yh}  
\\ &= \la^{-q'} \left[ \frac{1}{\rho q'} (1- e^{\sqrt 2 \rho q' a}) \right]^{\yh} \leq \frac{1}{q'} \la^{-q'} \rho^{-\yh}. 
\end{align*}
\\ \smallskip
\textit{Case 2: $|\gm| = \min \{ \la, |\gm|, \rho^{-\yh} \} $}
\begin{align*}
  I \leq \la^{-q'} \int_{\gm} e^{-\rho q' x^2} \, dx\leq \la^{-q'} |\gm|.    
\end{align*}
\\ \smallskip
\textit{Case 3: $\la = \min \{ \la, |\gm|, \rho^{-\yh} \} $}. It suffices to take $\gm = (\xi,\eta) \subset [0,\infty)$. 
\begin{align*}
  I \leq \int_{\gm} (|x|+\la)^{-q'} \, dx
= \int_\xi^\eta (x +\la)^{-q'} \, dx
= \frac{1}{q'-1} \left[ (\xi+\la)^{1-q'} - (\eta+\la)^{1-q'} \right] \leq \frac{1}{q'-1} \la^{1-q'}. 
\end{align*}

\end{proof}
\begin{prop} \label{Prop::Car_ineq} 
    Suppose $2d^2 -1 \leq q <\infty$ and let $\psi: \R^+ \to \R$ be $C^2$ and satisfies the following conditions: there is $C >0$ such that $C^{-1} < \psi'(\si) < C$ for all $\si$, and for any $\del >0$ there is $C_\del >0$ such that $\psi''(\si) \geq C_\del e^{-\del \si}$. 
    Let $1 < \mu < \frac{d}{d-1}$. Then for $\nu$ large enough and $\beta, \gm$ intervals with $\min \{ |\beta|, |\gm| \} \geq \nu^{-1}$, we have
\begin{equation} \label{Car_ineq} 
 \| 1_{A(\psi^{-1} \gm)} P_\nu   1_{A(\psi^{-1} \beta)} \|_{p \to p'} 
 \leq C_{d,p} \nu^{\frac{d-1}{q}+ \frac{1}{\mu} - \frac{1}{2\mu q'}} (|\gm|')^{\frac{1}{2q'}} .   
\end{equation}
\end{prop}
\begin{proof}
Write $P_\nu = Q_\nu + R_\nu$, where $Q_\nu$ (resp. $R_\nu$) comes from substituting $M_\nu$ (resp.$N_\nu$) for $L_\nu$ in the definition of $P_\nu$. Notice that $\nu = N+ \frac{d-1}{2} - \rho$ so $\nu \approx N$ when $\nu$ is large. We denote by $\Del_2 \psi (\si, \tau)$ the quantity $\psi(\tau) - \psi(\si) - \psi'(\si)(\tau -\si)$. Then we have 
\[
  |Q_\nu| = e^{-\nu \Del_2 \psi(\si, \tau)} |M_{\nu \psi'(\si)} | \lesssim |M_{\nu \psi'(\si)} | \lesssim  |z-\zeta|^{-(d-1)}.  
\]
Since $Q_\nu$ vanishes for $|z-\zeta| \geq N^{-1} \gtrsim \nu^{-1}$, we have 
\begin{align*}
  \| Q_\nu \|_{L^{q'}(B(0,1))}  
&= \left( \int_{|z-\zeta| < \nu^{-1} \cap B(0,1)}  |Q_\nu(z,\zeta)|^{q'} \, dV(\zeta) \right)^{\frac{1}{q'}}
  \\ &\lesssim \left( \int_{|z-\zeta|\lesssim \nu^{-1}\cap B(0,1) } |z-\zeta|^{-(d-1)q'} \, dV(\zeta) \right)^{\frac{1}{q'}}
\lesssim \nu^{\frac{d}{q}-1}.  
\end{align*}
It follows from \rl{Lem::op_norm_AB} that $Q_\nu$ is $L^p(B(0,1)) \to L^{p'}(B(0,1))$ bounded with norm 
\[
  \lesssim \nu^{\frac{d}{q} -1} = \nu^{\frac{d-1}{q}+ 
\frac{1}{q}-1} = \nu^{\frac{d-1}{q}- \frac{1}{q'}} \leq \nu^{\frac{d-1}{q}} (|\gm|')^{\frac{1}{q'}}  
 \leq \nu^{\frac{d-1}{q}+ \frac{1}{\mu} - \frac{1}{2\mu q'}} (|\gm|')^{\frac{1}{2q'}}. 
\]
Here we used $\nu^{-1} \leq  |\gm|' = \min \{ |\gm|, \nu^{-\yh} \}$. 
To finish the proof, it suffices to prove \re{Car_ineq} for $R_\nu$ in place of $P_\nu$. Choose a partition of unity on $S^{d-1} \times S^{d-1}$ consisting of functions $\{ \chi^{\nu^{-\frac{1}{\mu}}}, \chi_\la \}$, where $\la = 2^j \nu^{-\frac{1}{\mu}}$ for $0 \leq j \leq \log_2 (\nu^{\frac{1}{\mu}})$. Consider 
\[
  \left\| 1_{A(\psi^{-1} \gm)} (z) \chi_\la \left( \frac{z}{|z|}, \frac{\zeta}{|\zeta|} \right) R_\nu (z,\zeta) 1_{A(\psi^{-1} \beta)} (\zeta) \right\|_{p \to p'}, \quad 
  R_\nu (z,\zeta) = e^{-\nu \Del_2 \psi(\si, \tau)}  N_\nu (z,\zeta). 
\]
Let $r_\la(s,t)$ be the $L^p (S^{d-1}) \to L^{p'} (S^{d-1})$ norm of the kernel $(\chi_\la R_\nu)^{st}$.  
Let $T^\la$ be the operator defined by the kernel above. Then we can write
\begin{align*}
T^\la f (z) &= T^\la f(s,\xi)
= \int_{(\psi^{-1} \beta)_\ast} T^\la_{s,t} \left[ f(t, \cdot) \right] (\xi) t^{d-1} \, dt, 
\end{align*}
where $T^\la_{s,t} [f(t,\cdot)] (\xi) = \int_{S^{d-1}}(\chi_\la R_\nu)^{st} (\xi, \eta) f(t,\eta ) \, d \eta $. Regard $\R^d = S^{d-1} \times \R^d$. Applying \rl{Lem::op_norm_prod_space} and then \rl{Lem::op_norm_AB} with $u(s) = s^{-d}$ and $v(t) = t^{-d}$, we get
\begin{align*}
   \| T^\la \|_{p \to p'} 
   &\leq L^p (\beta_\ast, t^{d-1} \, dt) \to L^{p'} (\gm_\ast, s^{d-1} \, ds) 
 \: \text{norm of} \: r_\la(s,t)
\leq (AB)^{\yh}, 
\end{align*} 
where 
\begin{gather*}
  A = \sup_{s \in (\psi^{-1} \gm)_\ast } 
  \|e^{-\nu \Del_2 \psi(\si,\tau)} (st)^{d/p'} r_\la(s,t) \|_{L^{q'}\left((\psi^{-1} \beta)_\ast, \frac{dt}{t}\right) }; 
\\ 
B = \sup_{t \in (\psi^{-1} \beta)_\ast } 
\| e^{-\nu \Del_2 \psi(\si,\tau)} (st)^{d/p'} r_\la(s,t) \|_{L^{q'}\left( (\psi^{-1} \gm)_\ast, \frac{ds}{s} \right)}. 
\end{gather*}
Since $\la = 2^j \nu^{-\frac{1}{\mu}} \geq \nu^{-1}$, we have by \re{chi_la_Nst_nu_est},   
\begin{align*}
  r_\la(s,t) \lesssim \nu^{\frac{d}{2q} - \frac{1}{q}} \la^{-\frac{d}{2q}} (st)^{-\frac{d-1}{2}} K(s,t)
\end{align*}
where 
\[ 
  K(s,t) = \begin{cases} 
(| \si-\tau| + \la)^{-1} & \text{if $|\si - \tau|< 1$}; 
\\[5pt] 
e^{-\del |\si -\tau | }  & \text{if $|\si - \tau|> 1$}
\end{cases}
\]
and $\del = \dist (N-(d-1)/2, \Z)$. 
Since $\displaystyle \frac{1}{p} - \frac{1}{p'} = \frac{1}{q}$, we have $\displaystyle \frac{d}{p'} = \frac{d}{2q'} = \frac{d}{2} \left( 1- \frac{1}{q} \right)$ and 
\[
  \eta_{d,q} := \frac{d}{p'} -\frac{d-1}{2}= \yh \left(1-\frac{d}{q} \right)>0. 
\]
Hence 
\begin{equation} \label{A_est}  
   A \leq \nu^{\frac{d}{2q} - \frac{1}{q}} \la^{-\frac{d}{2q}} s^{\eta(d,q)} 
 \left( \int_{(\psi^{-1}\beta)_\ast} \left( e^{-\nu \Del_2 \psi (\si,\tau)} [K(s,t)] t^{\eta(d,q)} \right)^{q'} \frac{dt}{t} \right)^{\frac{1}{q'}}.  
\end{equation}
Convexity of $\psi$ implies that $\Del_2 \psi$ is positive and in fact $\geq k(\si)\min \{ 1, (\tau-\si)^2 \}$, where $k(\si)$ denotes any lower bound for $\psi''/2$ on the interval $(\si-1, \si +1)$. We can take $k(\si)$ to be of the form $C_\del e^{-   \del \si}$ for any given $\del>0$ such that 
\[
  \Del_2 \psi \geq C_\del e^{-\del \si} \min \{ 1, (\tau-\si)^2 \} = C_\del s^\del \min \{ 1, (\tau-\si)^2 \} .   
\] 

We estimate the integral in \re{A_est} by splitting into two parts: $\psi^{-1} \beta = \beta^1 \cup \beta^2$  \\ 
\textit{Region 1.} 
Denote $\beta^1 := \{ \tau \in \psi^{-1} \beta, \: |\si - \tau| > 1\}$. 
On $\beta^1_\ast$, we have $\Del_2 \psi \geq  C_\del s^\del$ and $K(s,t) = e^{-\del |\si - \tau|}$. Since $ 
\beta^1 \subset \R^+$, $e^{-\tau \eta(d,q)}$ is uniformly bounded on $\beta^1$.  
Using $t = e^{-\tau}$, the integral in \re{A_est} is bounded by 
\begin{align*} 
  e^{-C_\del \nu s^\del}  \left( \int_{\beta^1_\ast} \left( e^{-\del |\si -\tau| } t^{\eta(d,q)}\right)^{q'}  \frac{dt}{t} \right)^{\frac{1}{q'}} 
&= e^{-C_\del \nu s^\del} \left( \int_{\beta^1}  
\left( e^{-\del |\si -\tau| } e^{-\tau \eta(d,q)} \right)^{q'}\, d \tau \right)^{\frac{1}{q'}}  
  \\ &\lesssim e^{-C_\del \nu s^\del} 
\end{align*}
which can be made $\lesssim \nu^{-T}$ for any given $T$.
\\ 
\smallskip 
\textit{Region 2. } Denote $ \beta^2:= \{ \tau 
\in \psi^{-1} \beta, |\si - \tau| < 1 \}$. On $\beta^2_\ast$, we have $\Del_2 \psi \geq C_\del s^\del (\tau -\si)^2$ and $K(s,t) = (| \si-\tau| + \la)^{-1}$. Application of \rl{Lem::min_est} then gives 
\begin{align*}
    &s^{\eta(d,q)} \left( \int_{\beta^2_\ast} \left( e^{-C_\del \nu s^\del (\si-\tau)^2 } K(s,t) t^{\eta(d,q)} \right)^{q'} \frac{dt}{t} \right)^{\frac{1}{q'}} 
    \\ &\quad \leq s^{\yh(1-\frac{d}{q})} \left( \int_{\beta^2}  
\left[ e^{-C_\del \nu s^\del (\si-\tau)^2 } (| \si-\tau| + \la)^{-1} \, e^{-\tau \eta(d,q)} \right]^{q'} d \tau \right)^{\frac{1}{q'}}  
    \\ &\quad \lesssim s^{\yh(1-\frac{d}{q})} \la^{-1} \min \{ (\nu s^\del)^{-\yh}, \la,  |\beta^2| \}^{\frac{1}{q'}} 
    \\ &\quad \lesssim \la^{-1} \min \{ \nu^{-\yh}, \la, |\psi^{-1} \beta| \}^{\frac{1}{q'}} \lesssim \la^{-1} \min \{ \nu^{-\yh}, \la, |\beta| \}^{\frac{1}{q'}} , 
\end{align*}
provided that $\del \leq 1 - d/r$. Note that the last inequality is true since $\psi'$ is bounded away from $0$ and $\infty$. 
It follows from \re{A_est} that 
\[
  A \leq \nu^{\frac{d}{2q} -\frac{1}{q} } \la^{-\frac{d}{2q} -1} \min \{ \nu^{-\yh}, \la, |\beta| \}^{\frac{1}{q'}}.    
\]
Similarly, we can estimate $B$ with the same bound with $|\beta|$ replaced by $|\gm|$.   
Summing over $\la$ we obtain 
\begin{align*}
  \| 1_{A(\psi^{-1} \gm)} R_\nu 1_{A(\psi^{-1} \beta)} \|_{p \to p'} 
&\leq \| 1_{A(\psi^{-1} \gm)} \chi^{\nu^{\frac{1}{\mu}}} R_\nu  1_{A(\psi^{-1} \beta)} \|_{p \to p'} + \sum_{\substack{\la=2^j \nu^{-\frac{1}{\mu}} \\[1.5pt] 0 \leq j \leq 
\log_2 \nu}} \| 1_{A(\gm)} \chi_\la R_\nu  1_{A(\beta)} \|_{p \to p'} 
\end{align*}   
\begin{equation} \label{chi_la_Rnu_est}
 \sum_{\substack{\la=2^j \nu^{-\frac{1}{\mu}} \\[1.5pt] 0 \leq j \leq 
\log_2 \nu}} \| 1_{A(\psi^{-1} \gm)} \chi_\la R_\nu  1_{A(\psi^{-1} \beta)} \|_{p \to p'} 
\leq \nu^{\frac{d}{2q} -\frac{1}{q} } 
\sum_{\substack{\la=2^j \nu^{-\frac{1}{\mu}} \\[1.5pt] 0 \leq j \leq 
\log_2 \nu}} \la^{-\frac{d}{2q} -1} \left[ \min \{ \nu^{-\yh}, \la, |\beta| \} \cdot \min \{ \nu^{-\yh}, \la, |\gm| \} \right]^{\frac{1}{2q'}}.    
\end{equation}
Note that the power of $\la$ is always negative. 
We use 
\[ 
 \min \{ \nu^{-\yh}, \la, |\beta| \} \leq \la, \quad 
 \min \{ \nu^{-\yh}, \la, |\gm| \} \leq |\gm|',  
\] 
so that the above sum is bounded by 
\begin{align*}
(|\gm|')^{\frac{1}{2q'}} \sum_{\substack{\la=2^j \nu^{-\frac{1}{\mu}} \\[1.5pt] 0 \leq j \leq 
\log_2 \nu}} \la^{-1 - \frac{d}{2q} + \frac{1}{2q'}} 
&= (|\gm|')^{\frac{1}{2q'}} \sum_{j \geq 0} \left(2^j \nu^{-\frac{1}{\mu}} \right)^{-1 - \frac{d}{2q}+ \frac{1}{2q'}}
\\ &\lesssim (|\gm|')^{\frac{1}{2q'}} \nu^{\frac{1}{\mu} + \frac{d}{2\mu q}- \frac{1}{2\mu q'}}. 
\end{align*}
Hence \re{chi_la_Rnu_est} implies 
\begin{align*}
\sum_{\la} \| 1_{A(\psi^{-1} \gm)} \chi_\la R_\nu  1_{A(\psi^{-1} \beta)} \|_{p \to p'} 
&\leq \nu^{\frac{d}{2q} - \frac{1}{q}}  (|\gm|')^{\frac{1}{2q'}} \nu^{\frac{1}{\mu} + \frac{d}{2\mu q} - \frac{1}{2\mu q'}} 
\leq \nu^{\frac{d-1}{q}+ \frac{1}{\mu} - \frac{1}{2\mu q'}} (|\gm|')^{\frac{1}{2q'}}   
\end{align*}
where in the last step we used $\displaystyle \nu^{\frac{d}{2q} + \frac{d}{2\mu q}} \leq \nu^{\frac{d}{q}} $. 

We now estimate $ \| 1_{A(\psi^{-1}\gm)} \chi^{\nu^{-\frac{1}{\mu}}} R_\nu  1_{A(\psi^{-1} \beta)} \|_{p \to p'} $. Let $r_0(s,t)$ be the $L^p (S^{d-1}) \to L^{p'} (S^{d-1})$ norm of the kernel $(\chi^{\nu^{-\frac{1}{\mu}}} R_\nu)^{st}$. Then by \rl{Lem::op_norm_prod_space} and \rl{Lem::op_norm_AB} we have 
\[
  \| 1_{A(\psi^{-1}\gm)} \chi^{\nu^{-\frac{1}{\mu}}} R_\nu  1_{A(\psi^{-1} \beta)} \|_{p \to p'} \leq (AB)^\yh, 
\]
where
\begin{gather*}
  A = \sup_{s \in (\psi^{-1} \gm)_\ast } 
  \|e^{-\nu \Del_2 \psi(\si,\tau)} (st)^{d/p'} r_0(s,t) \|_{L^{q'}\left((\psi^{-1} \beta)_\ast, \frac{dt}{t}\right) }; 
\\ 
B = \sup_{t \in (\psi^{-1} \beta)_\ast } 
\| e^{-\nu \Del_2 \psi(\si,\tau)} (st)^{d/p'} r_0(s,t) \|_{L^{q'}\left( (\psi^{-1} \gm)_\ast, \frac{ds}{s} \right)}. 
\end{gather*}
Using \rl{chi^la_Nst^nu_est}, we obtain  
\[
  r_0(s,t) \lesssim  \nu^{\frac{d}{2 \mu q} - \frac{1}{\mu q}} \la^{-\frac{d}{2q} + (d-1) - \mu(d-2) } (st)^{-\frac{d-1}{2}} \times 
  \begin{cases} 
(| \si-\tau| + \la^\mu)^{-1} & \text{if $|\si - \tau|< 1$}; 
\\[5pt] 
e^{-\del |\si -\tau | }  & \text{if $|\si - \tau|> 1$}. 
\end{cases} 
\]
Estimating $A$ and $B$ in the same way as in \re{A_est}, we have
\begin{align*}
  \| 1_{A(\gm)} \chi^{\nu^{-1}} R_\nu  1_{A(\beta)} \|_{p \to p'}
&\lesssim \nu^{ \frac{d}{2 \mu q} - \frac{1}{ \mu q}} 
 \la^{-\frac{d}{2q} +(d-1)- \mu (d-2) -\mu} \left[ \min \{ \nu^{-\yh}, \la^\mu, |\beta| \} \cdot \min \{ \nu^{-\yh}, \la^\mu, |\gm| \} \right]^{\frac{1}{2q'}}
\\ &\lesssim  \nu^{ \frac{d}{2 \mu q} - \frac{1}{\mu q}} 
 \la^{-\frac{d}{2q} - (\mu-1 )(d-1)} \left[ \min \{ \nu^{-\yh}, \la^\mu, |\beta| \} \cdot \min \{ \nu^{-\yh}, \la^\mu, |\gm| \} \right]^{\frac{1}{2q'}}
 \\ &\lesssim \nu^{ \frac{d}{2q} - \frac{1}{q}} 
 \la^{-\frac{d}{2q} - 1} \left[ \min \{ \nu^{-\yh}, \la, |\beta| \} \cdot \min \{ \nu^{-\yh}, \la, |\gm| \} \right]^{\frac{1}{2q'}}
\end{align*} 
where we take $\mu\geq 1$ to satisfy 
\begin{equation} \label{mu_cond} 
  (\mu-1)(d-1) < 1, \quad \text{or} \quad \mu < \frac{d}{d-1}.    
\end{equation}
Same estimate as for \re{chi_la_Rnu_est} then yields 
\[ 
   \| 1_{A( \psi^{-1} \gm)} \chi^{\nu^{-1}} R_\nu  1_{A(\psi^{-1} \beta)} \|_{p \to p'} \lesssim \nu^{\frac{d-1}{q}+ \frac{1}{\mu} - \frac{1}{2\mu q'}} (|\gm|')^{\frac{1}{2q'}}. \qedhere 
\]
\end{proof}

\begin{cor} \label{Cor::Car_ineq}    
  Let $q \geq 2d^2 -1$. Then $ \| 1_{A(\psi^{-1} \gm)} P_\nu 1_{B(0,1)} \|_{p \to p'} \leq C_{d,p} \nu \min 
  \{ |\gm|,\nu^{-\yh}\} $. 
\end{cor}
\begin{proof}
We denote $|\gm|' = \min\{ 
|\gm|, \nu^{-\yh} \}$. Using \rp{Prop::Car_ineq} with $\beta = \infty$ 
($|\beta|' = \nu^{-\yh}$), we get 
\begin{align*} 
   \| 1_{A(\psi^{-1} \gm)} P_\nu \|_{p \to p'} 
&\lesssim  \nu^{\frac{d-1}{q}+ \frac{1}{\mu} - \frac{1}{2\mu q'}} (|\gm|')^{\frac{1}{2q'}}   
\\ &= (\nu |\gm|' ) \left(  \nu^{\frac{d-1}{q}+ \frac{1}{\mu} - \frac{1}{2 \mu q' } -1}  (|\gm|')^{\frac{1}{2q'} -1 } \right)
\\ &\leq (\nu |\gm|' ) \left(  \nu^{\frac{d-1}{q}+ \frac{1}{\mu} - \frac{1}{2 \mu q' } - \frac{1}{2q'}}  \right)
\end{align*} 
where we used that $|\gm|' \geq \nu^{-1}$. 
Now, 
\begin{align*}
 \frac{d-1}{q}+ \frac{1}{\mu} - \frac{1}{2 \mu q' } - \frac{1}{2q'}
&= \frac{d}{q} + \left( 1 - \frac{1}{q} - \frac{1}{2q'} - \frac{1}{2q'} \right) - \left(1-\frac{1}{\mu} \right) + \left( \frac{1}{2q'} - \frac{1}{2\mu q'} \right)  
\\ &= \frac{d}{q} - \left( 1 - \frac{1}{2q'} \right) \left( 1- \frac{1}{\mu} \right). 
\end{align*}
The above expression is $\leq 0$ if (taking into account \re{mu_cond})
\[
 q \geq \frac{2d}{ 1- \mu^{-1}} -1 = 2d^2 -1.  \qedhere 
\]
\end{proof}
The following lemma of Wolff is the key to deal with the blow-up of the constant in Carleman inequality. 
\begin{lemma} \label{Lem::measure}  
  Suppose $\mu$ is a positive measure with faster than exponential decay 
\begin{equation} \label{meas_decay}  
    \lim_{T \to \infty} \frac{1}{T} \log \mu \{ x: |x| > T \}  = -\infty.   
\end{equation}
Define $\mu_k$ for $k \in \R$ by $d \mu_k (x) = 
e^{kx} d \mu(x)$. Suppose $N \in \R^+$. Then there are disjoint intervals $I_j \subset \R$ and numbers $k_j \in [N,2N]$ such that (with $C$ a positive universal constant) 
\begin{enumerate}[(i)]
\item 
$\mu_{k_j} (I_j) \geq \yh \| \mu_{k_j} \|$; 
\item $\sum |I_j|^{-1} \geq CN$. 
\end{enumerate} 
\end{lemma} 
\begin{rem}
 Condition \re{meas_decay} means that for any constant $C>0$, 
 \[ 
\lim_{T \to \infty} \frac{1}{CT} \log \mu \{ x: |x| > T \}  = -\infty.
\] 
Taking exponential on both sides we get 
\[
  \lim_{T \to \infty} \mu \{ x: |x| > T \}^{\frac{1}{CT}} = 0. 
\]
In other words, for each $M>0$, we have $\mu \{ x: |x| > T \} \leq M^{-CT}$ for all $T>T(M)$. This explains the meaning of ``faster than exponential decay". 
\end{rem}

We now extend formula \re{Pnu_formula_CCinfty} to all functions $f$ vanishing to infinite order at the origin.  
\begin{prop} \label{Prop::int_formula_W1p}  
 Let $f \in W^{1,p}(B(0,1))$ with $\supp f \in B(0,1)$. Suppose that $\| f \|_{L^{p}(B(0,r))}$ and $\| \db f \|_{L^{p}(B(0,r))}$ vanish faster than any powers of $r$ as $r \to 0$. 
Then the following formula holds
\begin{equation} \label{int_formula_W1p}  
  e^{\nu \psi (\si)} f (z) = \int_{\C^n} P_\nu (z,\zeta) e^{\nu \psi(\tau)} \we \db f(\zeta),   
\end{equation}
where both sides of the equation are in $L^p(B(0,1))$. 
\end{prop} 
\begin{proof}
First, suppose $f$ is supported away from $0$. Let $\psi_\ve$ be the standard mollifier. Then for $\ve>0$ sufficiently small, $\psi_\ve \ast f \in C^\infty_c (B(0,1)\sm \{0 \})$ and thus 
\begin{equation} \label{Pnu_formula_psi_ve_ast_f}   
  e^{\nu \psi (\si)} \psi_\ve \ast f (z) = \int_{\C^n} P_\nu(z,\zeta) e^{\nu \psi(\tau)} \we \db (\psi_\ve \ast f) 
  = \int_{\C^n} P_\nu(z,\zeta) e^{\nu \psi(\tau)} \we  
(\psi_\ve \ast \db f).  
\end{equation} 
Recall that 
\[
  P_\nu(z,\zeta) = e^{-\nu \Del_2 \psi (\si,\tau)} L_\nu (z,\zeta) = e^{-\nu \Del_2 \psi (\si,\tau)} (|z| |\zeta|)^{-\frac{d-1}{2} + \rho } I^N (z,\zeta), 
\] 
where $\Del_2 \psi(\si,\tau) = \psi(\tau) - \psi(\si) - \psi'(\si) (\tau - \si)$. Since $\psi_\ve \ast \db f \in C^\infty_c ( \R^{2n} \sm \{0 \}) $ and $\psi' \leq C$, it follows that 
$e^{\nu \psi (\tau) - \nu \Del_2 \psi(\si,\tau)} 
(|z| |\zeta|)^{-\frac{d-1}{2} +\rho}$ is uniformly bounded on $\supp 
(\psi_\ve \ast \db f)$. 
Let $\ve \to 0$ on both sides of \re{Pnu_formula_psi_ve_ast_f}. On the left we have $\| \psi_\ve f - f \|_{L^p(\R^{2n})} \to 0$. Now, 
\[
  I^N (z,\zeta) =|z|^{-N} |\zeta|^{N+d-1} \left( B(z,\zeta) - P^{N-1}(z,\zeta) \right).  
\]
Let $g_\ve = \db f \ast \psi_\ve - \db f$. Then 
$g_\ve$ is $L^p$ with compact support in $\R^{2n} \sm \{0 \}$ and $\| g_\ve \|_{L^p(\R^{2n})} \to 0$. Since $|B_j(z,\zeta)| = \left| \frac{\ov{\zeta_j -z_j}}{|\zeta-z|^{2n}} \right| \lesssim \frac{1}{|\zeta-z|^{2n-1}}$, we get 
\begin{equation} \label{P_N-1_int_norm}
 \left\| \int P_\nu(\cdot,\zeta) e^{\nu \psi(\tau)} \we g_\ve
\right\|_{L^p(B(0,1))} \lesssim 
\left\| \int \frac{1}{|\zeta-\cdot|^{2n-1}} |g_\ve(\zeta)|   \right\|_{L^p(B(0,1))} + \left\| \int |P^{N-1}(\zeta,z)| |g_\ve(\zeta)|   \right\|_{L^p(B(0,1))}. 
\end{equation}

For the first integral on the right, we can write it as $\left\| \frac{1}{|\zeta - \cdot |^{2n-1}} \ast g_\ve \right\|_{L^p(B(0,1))}$ and apply Young's inequality to obtain 
\[
  \left\| \frac{1}{|\zeta - \cdot |^{2n-1}} \ast g_\ve \right\|_{L^p(B(0,1))}
 \lesssim \left\| \frac{1}{|\zeta|^{2n-1}}  \right\|_{L^1(B(0,R))} \| g_\ve \|_{L^p (\R^{2n})} 
\longrightarrow 0, \quad \text{as $\ve \to 0$},  
\]
where $B(0,R)$ is some sufficiently large ball. 

For the last term in \re{P_N-1_int_norm}, we note that since $P^{N-1}(\zeta,\cdot)$ is the degree $N-1$ polynomial of $ \frac{\ov{\zeta_j -z_j}}{|z - \zeta|^{2n} }$ at $z =0$, it is bounded on the support of $g_\ve$ and the bound is uniform for all $\ve$ sufficiently small. Thus 
\[ 
\left\| \int |P^{N-1}(\zeta,z)| |g_\ve(\zeta)|   \right\|_{L^p(B(0,1))} \lesssim \| g_\ve \|_{L^1(\R^{2n})}  \longrightarrow 0, \quad \text{as $\ve \to 0$.}  
\]
This finishes the proof for the case $0 \notin \supp f$. 

Now, we show that formula \re{int_formula_W1p} still holds if $f \in W^{1,p}(B(0,1))$ and $0 \in \supp f \subset B(0,1)$. Take a smooth cutoff function $\chi_\ve$ such that $\chi_\ve(z) \equiv 0$ on $|z| < \ve$ and $\chi_\ve(z) \equiv 1$ on $|z| > 2 \ve$. Then we have 
\begin{equation} \label{chi_ve_f_int_formula} 
   e^{\nu \psi(\si)} \chi_\ve f(z) 
  = \int P_\nu (z,\zeta) e^{\nu \psi(\tau)} \we  \db (\chi_\ve f) .    
\end{equation}
Since $0 < \frac{1}{C} < \psi' \leq C$, we have $\psi (a) \leq \psi(b) \leq C' b$ for any $a<b$.  
First, we show that $ e^{\nu \psi(\si)} \chi_\ve f$ converges to $e^{\nu \psi(\si)} f$ in $L^p(\R^d)$. By the dominated convergence theorem it suffices to show $e^{\nu \psi(\si)} f \in L^p(\R^d)$. On each annulus $|z| \in (2^{-j}, 2^{-(j-1)}) $, $\psi(\log (|z|^{-1})) \leq \psi (\log (2^j)) \leq C' \log 2^j$. Thus
\begin{align*}
  \int_{B(0,1)} e^{p \nu \psi(\si)} |f|^p 
&= \sum_{j\geq 0} \int_{ 2^{-j} \leq |z| \leq 2^{-(j-1)}} e^{p \nu \psi(\log (|z|^{-1}))} |f(z)|^p
\\ &\leq \sum_{j\geq 0} 2^{j C' p \nu} \int_{ 2^{-j} \leq |z| \leq 2^{-(j-1)}}  |f(z)|^p 
\\ &\leq C_N \sum_{j\geq 0} 2^{j C' p \nu} 2^{-jN} <\infty
\end{align*}
where in the last step we use the $L^p$ infinite order vanishing of $f$ to choose $N > C' p \nu$ to make the sum converge.  

To finish the proof we show that 
\[
  g_\ve = \int P_\nu (\cdot,\zeta) e^{\nu \psi(\tau)} \we  \db (\chi_\ve f) \quad \text{converges to }  
  g = \int P_\nu (\cdot,\zeta) e^{\nu \psi(\tau)} \we  \db f \quad \text{in $L^p(B(0,1))$}.  
\] 
By the same proof as above and using the $L^p$ infinite order vanishing of $\db f$ at $0$, we have $e^{\nu \psi(\tau)} \db f \in L^p(\R^d)$. 
Write 
\[
 g_\ve (z) 
  = \int_{\R^d} P_\nu (z,\zeta) e^{\nu \psi(\tau)} \we 
(\db \chi_\ve) f +  \int_{\R^d} P_\nu (z,\zeta) e^{\nu \psi(\tau)} \we \chi_\ve \left( \db f \right):= g_\ve^1(z) + g_\ve^2(z). 
\]
By \rp{Prop::Car_ineq} we have 
\[
  \| g^1_\ve \|_{L^{p'}(B(0,1))} 
  \leq C_\nu \| e^{\nu \psi(\tau)} (\db \chi_\ve) f \|_{L^p(B(0,1))}.   
\]
Using $\supp \db \chi_\ve \subset (\ve, 2 \ve)$, and infinite order vanishing of $f$, we get 
\[
 \| g^1_\ve \|_{L^{p'}(B(0,1))} \lesssim  \ve^{-C\nu} \| f \|_{L^p(B(0,2\ve))} \to 0 \quad \text{as $\ve \to 0$}. 
\]
We now show that $g^2_\ve \to g$ in $L^p(B(0,1))$. By \rp{Prop::Car_ineq}, we have 
\[ 
\left\| \int P_\nu (\cdot,\zeta) e^{\nu \psi(\tau)} \we  (\chi_\ve \db f - \db f) \right\|_{L^{p'}(B(0,1))} \leq C_\nu \| e^{\nu \psi(\si)}  (\chi_\ve \db f  - \db f) \|_{L^{p}(\R^{2n})}. 
\]
Since $e^{\nu \psi(\si)} \db f \in L^p(\R^d)$, the last expression converges to $0$ when $\ve \to 0$ by the dominated convergence theorem. This shows that $g^2_\ve \to g$ in $L^p(B(0,1))$. Together with $g_\ve^1 \to 0$, this implies that $g_\ve \to g$ in $L^p(B(0,1))$.  \qedhere
\end{proof}
We are now ready to prove the unique continuation property for $|\db u| \leq V|u|$. 

\medskip 

\nid \textit{Proof of \rt{Thm::mt_intro} }. 
Choose $\phi \in C^\infty_c (B(0,s_0))$ with $\phi =1$ on $B((0,s_1))$. Let $\e_0$ be small enough and suppose $\| V \|_{L^q(B(0,s_0))} < \e_0$ but $u$ does not vanish identically on $B(0, s_0)$. Let $u = (u_1,\dots, u_n)$ be given as in the theorem. Using the vector notation $f=(f_1,\dots, f_n)$, we set $f= \phi u $. Then
\[
  \db f = u \db \phi + \phi \db u.  
\]
Since by assumption $\| u \|_{L^{p'}(B(0,r))}$ vanishes to infinite order as $r \to 0$, so does $ \| u \db \phi \|_{L^p(B(0,r))}$ ($ p < p'$). On the other hand, by H\"older's inequality applied to $\frac{1}{p} - \frac{1}{p'} = \frac{1}{q}$, we get 
\[
  \| \phi \db u \|_{L^p(B(0,r))} \leq \| \phi V u \|_{L^p(B(0,r))}
  \leq \| V \|_{L^q(B(0,r))} \| u \|_{L^{p'}(B(0,r))}. 
\]
Therefore $\| \phi \db u \|_{L^p(B(0,r))}$ vanishes to infinite order as $r \to 0$. This shows that $\| \db f \|_{L^p(B(0,r))}$ vanishes to infinite order. By \rp{Prop::int_formula_W1p} applied to each $f_\all$, we get  
\[ 
  e^{\nu \psi (\si)} f (z) = \int_{\R^d} P_\nu (z,\zeta) e^{\nu \psi(\tau)} \we \db f(\zeta).  
\]
Thus \rp{Prop::Car_ineq} gives 
\begin{equation} \label{f_Car_ineq} 
  \|  e^{\nu \psi (\si)} f\|_{L^{p'}(A(\psi^{-1} \gm))} 
\leq \nu \min \{ |\gm|,\nu^{-\yh}\} \|  e^{\nu \psi (\si)} \db f \|_{L^p(B(0,1))}     
\end{equation}
where we recall the notation $A(\psi^{-1} \gm)= 
\{ x \in \R^d: \psi (\log \frac{1}{|x|}) \in \gm \}$. 
Define 
\begin{align*}
   \mu (\gm) = \int_{A( \psi^{-1} \gm)} (V|f| )^p \, dy, \quad f = \phi u. 
\end{align*}
The infinite order vanishing of $\| u \|_{p}$ 
(and thus $\|f\|_p$) implies that $\mu$ has the faster-than-exponential decay property. Indeed, if $\psi (\log \frac{1}{|x|}) = \rho(x)$, then $|\rho(x)| > T$ implies $|x| < e^{-c_0 T}$ or $|x| > e^{c_1 T}$ for some constant $c_0, c_1 >0$. By H\"older inequality and using the fact that $f$ is compactly supported, we have for each fixed $M>0$, 
\[
 \mu \left( \{ x: |\rho(x) | >T \} \right) \leq \left(  \int_{|x| < e^{-c_0 T}} |V|^q \right)^\frac{p}{q} \left(  \int_{|x| < e^{-c_0 T}} |f|^{p'} \right)^\frac{p}{p'}   
  \leq C_M \| V \|^p_{L^q} e^{-c_0p M T}, \quad \text{for all $T>0$}.   
\]
Thus
\[
 T^{-1} \log   \mu\{ |\rho| >T \} 
 \leq T^{-1} \log \| V \|_q^p + T^{-1} \log C_M - c_0 p M.   
\]
The first term goes to $0$ as $T \to \infty$. We can assume that $C_M > 1$. Then for each fixed $\ve >0$ and $M$, we can choose $T$ sufficiently large that the second term is less than $\ve M$. Since this holds for arbitrary large $M$, we see that the sum of the terms of the above right-hand side goes to $-\infty$ as $T \to \infty$.

We can write $\mu$ as
\[
 \mu (\gm) =   \int_{[\psi^{-1}(\gm)]_\ast} \int_{S^{d-1}} 
(V| f| )^p(r\om) d \om \, r^{d-1} dr 
=  \int_\gm g_\ast \left[ \int_{S^{d-1}} (V|f| )^p(r \om) r^{d-1} d\om \, dr \right] .   
\]
Here $g: \R^+ \to \R$ is given by $g (t) = \psi( \log t^{-1})$.  
Thus \[ 
 d\mu  = g_\ast \left[ \int_{S^{d-1}} (V|f| )^p \, d\om r^{d-1} dr \right] 
 =  \int_{S^{d-1}} (V(\ti r\om) |f(\ti r\om)| )^p \, d\om \ti r^{d-1} \frac{\pa g}{\pa \ti r} d \ti r, \quad \ti r = g^{-1}(r). 
\]  Since $d \mu_k (x) = e^{kx} d \mu(x)$, we have 
\begin{align*}
 \mu_k (\gm) &=  \int_\gm e^{kx} \, d \mu(x) 
  = \int_\gm e^{kx} g_\ast \left[ \int_{S^{d-1}} (V|f| )^p \, d\om r^{d-1} dr 
\right] 
\\ &= \int_{A(\psi^{-1} (\gm))} e^{k g(y)} (V|f|)^p 
=  \int_{A( \psi^{-1} \gm)}  (e^{\nu 
\psi(\si)}V|f|)^p,   
\end{align*}
where we denote $\nu = k/p$.    
For $N$ sufficiently large, we let $\{ I_j \} $ and $k_j \approx N$ be from \rl{Lem::measure}. We may assume that for all $j$, $|I_j| \geq 1/N$ (otherwise discard all but one of them, and expand that one to length $1/N$). Note that $A(\psi^{-1} \R) = A(\R^+) = B(0,1)$.  Setting $\nu_j = k_j / p$, we apply H\"older's inequality and \re{f_Car_ineq} to get
\begin{equation} \label{mu_kj_norm_1} 
\begin{aligned}
  \| \mu_{k_j} \| 
  = \int_{A(\psi^{-1} \R)} (e^{\nu_j \psi(\si)} V |f|)^p 
&\leq 2 \int_{A(\psi^{-1} I_j) } (e^{\nu_j \psi(\si)} V | f|)^p
  \\ &\leq 2 \left( \int_{A(\psi^{-1} I_j) \cap B(0,s_0)} V^q \right)^{p/q} \left( \int_{A(\psi^{-1} I_j)} (e^{\nu_j \psi(\si)}|f|)^{p'}  \right)^{p/p'}.    
\end{aligned} 
\end{equation}
Now, 
\begin{align*}
 \left( \int_{A(\psi^{-1} I_j)} e^{p' \nu_j \psi(\si)} |f|^{p'} \right)^{p/p'} 
 &= \left(\int_{A(\psi^{-1} I_j)} e^{p' \nu_j \psi(\si)} \left( \sum_{\all=1}^m |f^\all|^2 \right)^{\frac{p'}{2}} \right)^{p/p'}
  \\ &\leq C_{p,m} \sum_{\all=1}^m 
 \left( \int_{A(\psi^{-1} I_j)} e^{p' \nu_j \psi(\si)} |f^\all|^{p'} \right)^{p/p'}.  
\end{align*}
Apply the Carleman inequality \re{f_Car_ineq} to each $f^\all$ to get
\begin{align*}
 \left( \int_{A(\psi^{-1} I_j)} e^{p' \nu_j \psi(\si)} |f|^{p'} \right)^{p/p'} 
& \leq C_{p,m} \left( \nu_j |I_j| \right)^p \sum_{\all =1}^m  \left( \int_{A(\psi^{-1} I_j)} e^{p \nu_j \psi(\si)} |\db f^\all|^{p} \right)  
\\& \leq C_{p,m} \left( \frac{k_j}{p} |I_j| \right)^p \left( \int_{A(\psi^{-1} I_j)} e^{p \nu_j \psi(\si)} |\db f|^{p} \right)
\\ &\leq C_{p,m} \left( N|I_j| \right)^p  \| e^{\nu_j \psi(\si)} \db f \|^p_{L^p(B(0,1))} 
\end{align*} 
where we used that $|\db f|^p = (\sum_{\all=1}^m |\db f^\all |^2)^\frac{p}{2} \geq C_{m,p} \sum_{\all=1}^m |\db f^\all |^p$. Hence it follows from \re{mu_kj_norm_1} that 
\begin{equation} \label{mu_kj_norm_2}  
  \| \mu_{k_j} \| 
  \lesssim  \| V \|^p_{L^q(A(\psi^{-1}I_j) \cap B(0,s_0))} \left( N |I_j| \right)^p \| e^{\nu_j \psi(\si)} \db f \|^p_{L^p(B(0,1))}. 
\end{equation}
Using the notation: $\db f = (\db f)_{\all=1}^m$, we calculate $\db f = \db (\phi u)$ by the product rule: 
\begin{align*}
   |\db f | &= |\phi \db u + u \db \phi | 
   \\ &\leq \phi V  |u| + |u \db \phi| 
   \\ &= V |f| + \var,  
\end{align*}
where $\var: = u \db \phi \in L^p$ is supported in $\{ x: s_1 < |x| < s_0 \}$. Since 
\begin{align*}
 \| e^{\nu_j \psi(\si)} V f \|^p_{L^p(B(0,1))} 
 = \int_{A(\psi^{-1}\R)} \left( e^{ \nu_j \psi(\si)} |V f| \right)^p = \| \mu_{k_j} \|, 
\end{align*}
we get from \re{mu_kj_norm_2} 
\begin{equation} \label{mu_k_est} 
 \| \mu_{k_j} \| 
\lesssim \| V \|^p_{L^q(A(\psi^{-1}I_j)\cap B(0,s_0))} (N|I_j|)^p \left( \| \mu_{k_j} \| + \| e^{\nu \psi(\si)} \var \|^p_{L^p(\R^d)} \right). 
\end{equation}

On $\supp \var $, one has $e^{\nu_j \psi (\log \frac{1}{|x|})} \leq e^{\nu_j \psi (\log \frac{1}{s_1})} $, so the last term is $O(e^{p\nu_j \psi(\si_1)} )$, $\si_1 = \log \frac{1}{s_1}$. On the other hand, for $y \in D(0,s_1)$, we have $\psi(\log |y|^{-1}) - \psi(\log s_1^{-1} ) 
\geq c_1 (\log |y|^{-1} - \log s_1^{-1})= c_1 \log \frac{s_1}{|y| }$, and 
\begin{align*} 
\frac{\| \mu_{k_j} \| }{e^{p \nu_j \psi(\si_1)}} 
&= \int_{\R^d} \left( e^{\nu_j \left[ \psi(\log |y|^{-1} 
) - \psi(\log s_1^{-1}) \right] } |f| \right)^p dy
\\ &\geq \int_{B(0,s_1)} e^{c_1 \nu_j \log \frac{s_1}{|y|}} |u(y)|^p dy , \qquad \text{$f = u$ on $D(0,s_1)$}, 
\end{align*} 
where the last term goes to $\infty$ as $\nu_j = k_j / p \approx N/p \to \infty$.    
Hence for large $N $ we can absorb the last term of \re{mu_k_est} into $\| \mu_{k_j} \|$. Taking $p^{th}$ root on both sides we get 
\[ 
    \| \mu_{k_j} \|^{1/p} \lesssim  \| V \|_{L^q(A(\psi^{-1}I_j) \cap B(0,s_0))}  (N |I_j|) \| \mu_{k_j} \|^{1/p}.
\] 
By assumption $\| \mu_{k_j} \| \neq 0$, so 
\[ 
 \| V \|_{L^q(A(\psi^{-1}I_j) \cap B(0,s_0))} 
\gtrsim (N |I_j|)^{-1}.
\] 
Now, as $I_j$ are all disjoint, so are $A(\psi^{-1} 
I_j)$. Summing over $j$ and using \rl{Lem::measure},  
\[
  \| V \|_{L^q(B(0,s_0))} \geq N^{-1} \sum_j (|I_j|)^{-1} \gtrsim 1,   
\]
which is a contradiction if $\e_0$ is small.

We need the following result due to Gong and Rosay \cite{G-R_07}. 
\begin{prop} \label{Prop::G-R}  
  Let $\Om$ be an open set in $\C^n$ and $f: \Om \to \C$ be a continuous function. If on $\Om \sm f^{-1}(0)$, $|\db f| \leq C |f|$, for some positive constant $C$, then $f^{-1}(0)$ is an analytic set.     
\end{prop}
\begin{lemma} 
\label{Lem::HP}
  Let $K$ be a compact subset of $\R^d$. Let $l$ and $p'$ be some positive number such that $n - lp' >0$. For each $\ve >0$, there exists some $\chi_{\ve} \in C^{\infty}_c (\R^d)$ with $\chi_{\ve} \equiv 1 $ in a neighborhood of $K$ and $\supp \chi_{\ve} \subset K_{\ve}$. Furthermore, for $ | \all | <  l$, 
 \begin{equation}
 \left| D^\all \chi_\ve \right|_{L^{p'}} \leq C_{\all,n} \ve^{ l - | \all| } \left( \La_{d-lp'}  (K) + \ve \right)^{\frac{1}{p'}}.  
 \end{equation} 
\end{lemma} 
\begin{proof}
 See \cite{S-Z_22}.  
\end{proof}

\begin{thm} \label{Thm::Linfty_ucp}   
 If $V \in L^\infty(\Om)$, then the differential inequality $|\db u| \leq V|u| $ has the strong unique continuation property in the following sense: Let $u \in W^{1,2}_\loc(\Om)$. If $|\db u| \leq V|u|$ and
\[
 \lim_{R \to 0} R^{-N} \int_{|x| < R} |u|^2 = 0
\]
for all $N$, then $u$ is $0$. 
\end{thm}
\begin{proof}
  By Sobolev embedding, we have $u \in L^{\frac{2d}{d-2}}_\loc$. The inequality $|\db u| \leq C |u|$ shows that $\db u \in L^{\frac{2d}{d-2}}$. Then $u \in W^{1,\frac{2d}{d-2}}$. Apply Sobolev embedding again we get $u \in W^{\frac{2d}{d-4}}$. Repeat the procedure we eventually obtain $u \in L^p$ for some $p > n$. Hence $u \in C^0$.   By \rp{Prop::G-R}, $\Ac= u^{-1}(0)$ is a complex analytic variety. If $\dim(\Ac) = 2n$ then we are done. Otherwise suppose $\dim (\Ac) \leq 2n-2$. Define 
\[
   V = \begin{cases}
  \frac{\db u}{u} & \text{on $\Om \sm \Ac$}; 
\\ 0 & \text{on $\Ac$}.  
   \end{cases}
\]
We show that $\db u = Vu $ on $\Om$ in the sense of distributions, i.e. 
\begin{equation} \label{dbu=Vu_dist}  
  \int u \we \db \phi = - \int (Vu) \we \phi, \quad \text{for all $(n,n-1)$ form $\phi$ with coefficients in $C^\infty_c(\Om)$}.    
\end{equation}
Let $\{ \wti \chi_\ve \}_{\ve >0}$ be the family of cut-off functions from \rl{Lem::HP}, $0 \leq \wti \chi_\ve \leq 1 $, $\wti \chi_\ve \equiv 1$ in a neighborhood of $\Ac$, and supp $\wti \chi_\ve \seq \Ac_\ve$. Furthermore, 
\begin{equation} \label{chi_ve_grad_est}   
  \left| \na \wti \chi_\ve \right|_{L^{p'}(\Om)} \leq C_{\all,n} \ve^{ l - 1 } \left( \La_{n-lp'}  (\Ac) + \ve \right)^{\frac{1}{p'}}, \quad n-lp' > 0. 
\end{equation} 
Let $ \chi_\ve:= 1 - \wti \chi_\ve$, so that $\chi_\ve \equiv 0$ near $\Ac$ and $\chi_\ve \equiv 1$ in $\Om \sm \Ac_\ve$. Moreover, $\chi_\ve$ converges to $1$ point-wise in $\Om \sm \Ac$ as $\ve \to 0$, and $\chi_\ve$ satisfies estimate \re{chi_ve_grad_est}. By multiplying with another cut-off function, we can assume that $\chi_\ve$ has compact support in $\Om$. Apply Stokes theorem to get
\begin{equation} \label{ibp_chi_ve} 
  -\int_\Om Vu \we (\chi_\ve \phi) = \int_\Om u \we \db (\chi_\ve \phi) = \int_\Om u (\db \chi_\ve) \we \phi 
  + \int_\Om u \chi_\ve \we \db \phi .   
\end{equation}
Using $u \in L^1_\loc(\Om)$, $V \in L^\infty(\Om)$, as well as $\chi_\ve$ converges to $1$ a.e. on $\Om$, we get by the dominated convergence theorem 
\[
  \lim_{\ve \to 0} \int_\Om Vu \we (\chi_\ve \phi) = \int_\Om V u \we \phi, \qquad  \lim_{\ve \to 0} \int_\Om u \chi_\ve \db \phi = \int_\Om u \db \phi. 
\]
For the other integral we apply H\"older's inequality,
\begin{align*}
  \int_{\Om} \left| u (\db \chi_\ve) \we \phi \right| 
  &\lesssim  \| u \|_{L^p(\Om)} \| \db \chi_\ve \|_{L^{p'}(\Om)}
  \\ &\leq C_{\all,n} \ve^{l- 1} \left( \La_{n-lp'} (\Ac) + \ve \right)^{\frac{1}{p'}} \| f \|_{L^p(\Om)}. 
\end{align*}   
As $u \in C^0(\Om)$, we can choose any $p>2$ and $p'<2$ in the above estimate. It is then possible to choose $l>1$ such that $lp'<2$, so $n-lp'> n-2$. By the assumption $\dim \Ac \leq n-2$, we have $\La_{n-lp'}(\Ac)=0$, and therefore
\[
 \int_{\Om} \left| u (\db \chi_\ve) \we \phi \right| 
 \leq C_{\all,n} \ve^{l-1+\frac{1}{p'}} \| u \|_{L^p(\Om)}, \quad p > 2. 
\] 
In particular the integral converges to $0$ as $\ve \to 0$. Letting $\ve \to 0$ in \re{ibp_chi_ve} we obtain \re{dbu=Vu_dist}. 

Next, we show that $V$ is a $\db$-closed $(0,1)$ form in the sense of distributions. This amounts to showing that 
\begin{equation} \label{V_dbar_closed}
  \int_\Om V \we \db \phi = 0, \quad \text{ for all $(n,n-2) $ form $\phi$ with coefficients in  $C^\infty_c(\Om)$.}      
\end{equation} 
We first note that $V = \db \log u$ is $\db$-closed in the sense of distributions on $\Om \sm \Ac$. 
Let $\chi_\ve$ be the same cut-off function as above. Then we have 
\[
0 = \int_\Om V \we \db (\chi_\ve \phi) 
= \int_\Om V \we (\db \chi_\ve) \phi 
+ \int_\Om V \we \chi_\ve \db \phi. 
\]
By the same reasoning as above, the first integral converges to $0$ as $\ve \to 0$ and the second integral converges to $\int V \we \db \phi$. This proves \re{V_dbar_closed} and thus $V$ is $\db$-closed.   

We now solve $\db f = V$ on a small ball containing $0$. The argument here is standard. Write $V = \sum_{j=1}^n V_j d \ov z_j$. Let $V_\ve = \sum_{j=1}^n (V_j \ast \psi_\ve) d \ov z_j$, for a family of smooth mollifier function $\psi_\ve$.   Since $V$ is a $\db$-closed $(0,1)$ form with coefficients in $L^p$, for any $p <\infty$,  $V_\ve$ is well-defined and $\db$-closed in $\Om' \Subset \Om$ for small and positive $\ve>0$.  Moreover, $\| V_\ve - V \|_{L^p(\Om)} \to 0$. Let $\chi$ be a smooth function supported in $\Om'$ and $\chi \equiv 1$ on a small ball $B$ containing the origin. Apply the Bochner-Martinelli formula to get 
\[
  \chi V_\ve = \db \left( K_0 \ast (\chi V_\ve) \right) + K_1 \ast (\db \chi \we V_\ve).  
\] 
By standard Calderon Zygmund theory, 
$\| K_0 \ast (\chi (V_\ve -V) ) \|_{W^{1,p}(\Om)} \leq C_p \| V_\ve - V \|_{L^p(\Om)} \ra 0$ as $\ve \to 0$. On $B$, $K_1 \ast (\db \chi \we V_\ve)$ is smooth and $\db$-closed for small $\ve >0$. 
Hence there exists a smooth function 
$\ti u_\ve$ such that $\db \ti u_\ve = K_1 \ast (\db \chi \we V_\ve)$ on $B$, and $\| \wti u_\ve \|_{W^{1,p}(B)} \leq C_p \| K_1 \ast (\db \chi \we V_\ve) \|_{W^{1,p}(B)} $. Also $\| K_1 \ast (\db \chi \we  (V_\ve -V) ) \|_{W^{1,p}(B)} \leq C_p \| V_\ve - V \|_{L^p(\Om)} \to 0$ as $\ve \to 0$. 
We now have $\db ( K_0 \ast (\chi V_\ve) - \ti u_\ve) = V_\ve $ on $B$. Since $K_0 \ast (\chi V_\ve) - \ti u_\ve$ is a bounded sequence in $W^{1,p}(B)$, by weak compactness there exists a subsequence which converges to $f \in W^{1,p}(B)$ and $\db f = V$ on $B$.    
  
Let $h = u e^{-f}$. Then it is straightforward to show that the following holds in the sense of distributions: 

\[
  \db h = (\db u) e^{-f} -  u e^{-f} \db f 
  = e^{-f} (\db u - Vu) = 0, \quad \text{on $\Om$.} 
\]
This shows that $h$ is holomorphic. Since $u$ vanishes to infinite order at $0$ in the $L^2$ sense and $f$ is bounded, $h$ vanishes to infinite order in the $L^2$ sense at $0$. Estimating the derivatives of $h$ at $0$ by the value of $h$ in some disk shrinking to $0$, we see that all the derivatives of $h$ must vanish at $0$. Thus $h \equiv 0$, which implies that $u \equiv 0$. 
\end{proof}
\begin{rem}
    The proof of \rt{Thm::Linfty_ucp} can also be derived more directly from Gong and Rosay's argument. They actually first showed that $V$ is $\db$-closed, and then they used this fact to find $h$ by solving $\db$ equation and hence $u^{-1}(0) = h^{-1}(0)$ is the zero set of holomorphic function. 
\end{rem}
\bibliographystyle{amsalpha}
\bibliography{reference}  
\end{document}